\documentclass[12pt,twoside,reqno]{amsart}
\usepackage{verbatim,amscd,epsfig,amsmath,amsthm,amstext,amssymb,extpfeil,todonotes,enumerate,hyperref}
\usepackage[all]{xy}
\theoremstyle{plain}
\newtheorem{Theorem}{Theorem}[section]
\newtheorem{Lemma}[Theorem]{Lemma}
\newtheorem{Corollary}[Theorem]{Corollary}
\newtheorem{Proposition}[Theorem]{Proposition}
\newtheorem{Remark}[Theorem]{Remark}
\newtheorem{Definition}[Theorem]{Definition}
\newtheorem{Example}[Theorem]{Example}
\DeclareMathOperator{\Ann}{Ann}
\DeclareMathOperator{\supp}{supp}

\DeclareMathOperator{\kernel}{ker}
\DeclareMathOperator{\Res}{Res}
\newcommand{\unity}{{1\!\!\!\:\mathrm{l}}}

\newcommand{\Lss}{\mathcal L}
\newcommand{\Mss}{\mathcal M}

\newcommand{\Oss}{\mathcal O}

\newcommand{\Rss}{\mathcal R}
\newcommand{\Sss}{\mathcal S}

\newcommand{\Uss}{\mathcal U}

\newcommand{\C}{\mathbb{C}}

\newcommand{\Z}{\mathbb{Z}}

\renewcommand{\P}{\mathbb{P}^1}

\DeclareMathOperator*{\Pic}{Pic}

\newcommand{\ind}[1]{_{\mbox{\scriptsize{#1}}}}
\DeclareMathAlphabet{\Mat}{U}{eur}{m}{n} 
\DeclareMathAlphabet{\Set}{U}{eur}{m}{n} 
\newcommand{\Sh}[1]{\mathcal{#1}} 

\renewcommand{\qed}{{\bf\hspace{\fill}q.e.d.}}

\begin{document}
\author[S.~Klein]{Sebastian Klein}
\email{s.klein@math.uni-mannheim.de}
\address{Mathematics Chair III\\
Universit\"at Mannheim\\
D-68131 Mannheim, Germany}
\author[E.~L\"ubcke]{Eva L\"ubcke}
\email{eva.luebcke@gmail.com}
\address{Mathematics Chair III\\
Universit\"at Mannheim\\
D-68131 Mannheim, Germany}
\author[M.~Schmidt]{Martin Ulrich Schmidt}
\email{schmidt@math.uni-mannheim.de}
\address{Mathematics Chair III\\
Universit\"at Mannheim\\
D-68131 Mannheim, Germany}
\author[T.~Simon]{Tobias Simon}
\email{tsimon@mail.uni-mannheim.de}
\address{Mathematics Chair III\\
Universit\"at Mannheim\\
D-68131 Mannheim, Germany}
\title{Singular curves and Baker-Akhiezer functions}
\begin{abstract}
We present the concept of Baker-Akhiezer functions on singular complex curves. For this purpose, we translate the algebraic presentation of such curves in~\cite[Chapter~IV]{Se} into the analytic setting. Generalised divisors and their interplay with partial desingularisations are the fundament of the construction of Baker-Akhiezer functions.
\end{abstract}
\maketitle

\section{Introduction}
The main purpose of this article is a presentation of the theory of Baker-Akhiezer functions on singular one-dimensional complex spaces. These functions are defined in such a way that they describe sections of holomorphic line bundles on compact Riemann surfaces. Here, the line bundles are defined in terms of combinations of two concepts which both describe holomorphic line bundles on Riemann surfaces: divisors and cocycles. These concepts are combined in such a way that the holomorphic sections are uniquely determined and solve a differential equation. This relation between holomorphic sections and differential equations was first discovered by Burchnall and Chaundy in a series of papers on commutative algebras of ordinary differential equations. James Baker extracted in a note on one of these papers~\cite{Baker} a first definition of these functions. Krichever realized the importance of this concept and called them Baker-Akhiezer functions~\cite{Krichever}. For a general discussion of these functions, we recommend~\cite[Chapter~2 \S2]{DKN}. Our main objective is to construct Baker-Akhiezer functions on compact one-dimensional complex spaces in the most general setting.

Baker-Akhiezer functions are non-algebraic. Consequently, our main purpose requires the analytic description of compact one-dimensional complex spaces. To our knowledge, the literature on singular curves almost exclusively takes the algebraic point of view. Thus Section~\ref{Se:curves} contains the translation of the elegant treatment of algebraic singular curves by Serre \cite{Se} to the analytic setting. In Sections~\ref{se:riemann roch} and \ref{se:serre duality} the Riemann-Roch Theorem and Serre Duality are proven. The classical versions of these statements concern divisors on compact Riemann surfaces. On singular complex spaces, there exist several generalisations of the concept of divisors on smooth Riemann surfaces. We try to extend the concept of divisors to the most general situation in which Baker-Akhiezer functions are defined on singular one-dimensional complex spaces. This seems to be the concept of generalised divisors introduced by Hartshorne~\cite{Ha86}, which are defined as coherent subsheaves of the sheaf of meromorphic functions.
We hope to convince the reader that this concept is indeed natural on singular curves. Serre \cite[Chapter IV, \S 2--3]{Se} considers only the sub-class of invertible sheafs, which are locally free coherent sheaves of rank one. We again translate Serre's results to the analytic setting in Sections~\ref{se:riemann roch} and \ref{se:serre duality}.

In contrast to classical divisors, generalised divisors do not determine the underlying singular curves uniquely. For a given generalised divisor $\Sss'$, we discuss the compatible partial desingularisations of the underlying singular curve $X'$. Among these, there is a unique singular curve with lowest $\delta$-invariant, which we call $\Sss'$-halfway normalisation of $X'$. If the generalised divisor is locally free, then the underlying singular curve must be equal to this partial desingularisation.  

We finish this introduction with a short summary of the article. In Section~2, we translate~\cite[Chapter~IV \S1]{Se} into the analytic setting and describe singular one-dimensional complex spaces by their normalisation and some additional data. Generalised divisors are introduced in Section~\ref{se:generalised divisors} and their interplay with partial desingularizations of the corresponding singular curves are investigated in Section~\ref{se:halfway}. Sections~\ref{se:riemann roch} and \ref{se:serre duality} translate \S2 and \S3 of \cite[Chapter~IV]{Se} into the analytic setting. In Section~\ref{se:krichever}, we prepare the introduction of Baker-Akhiezer functions and recall Krichever's presentation of holomorphic line bundles on curves. Finally, in Section~\ref{se:BA} the Baker-Akhiezer functions are constructed.
\section{One-dimensional complex analytic spaces}\label{Se:curves}
\subsection{Structure of a one-dimensional complex analytic space}\label{sec:1}
Let $X'$ be a one-dimensional complex analytic space and $\Oss_{X'}$ the sheaf of the holomorphic functions on $X'$. The sheaf $\mathcal{M}_{X'}$ of meromorphic functions on $X'$ is defined as the sheaf of quotients $\frac{f}{g}$, where $f,g\in\Oss_{X'}$ and $g$ has isolated roots. For every $q\in X'$, we denote by $\Bar{\Oss}_{X',q}$ the integral closure of  $\Oss_{X',q}$ in  $\mathcal{M}_{X',q}$. We omit the subscript $X'$ if it is clear from the context on which space the sheaf is defined. This is the first step in constructing the normalisation $\pi:X\to X'$, see e.g. \cite[Theorem~4.4.8]{DJo}. In fact, the direct image $\pi_\ast\Oss_X$ of the sheaf of holomorphic functions $\Oss_X$ of $X$ equals $\Bar{\Oss}_{X'}$ which is the sheaf of locally bounded meromorphic functions on $X'$, see \cite[Theorem~4.4.15 and proof of Theorem 4.4.8]{DJo}:
\begin{align}
\label{eq:decompose-Obar}
\Bar{\Oss}_{X'}&=\pi_\ast\Oss_X\;,&
\Bar{\Oss}_{X',q}&:=\bigoplus_{p\in\pi^{-1}[\{q\}]}\Oss_{X,p} \; . 
\end{align}
The annihilator $\mathfrak{c}$ of $\Bar{\Oss}_{X'}/\Oss_{X'}$ with
\begin{equation}\label{eq:ann}
\mathfrak{c}_q:=\{g\in\Bar{\Oss}_q\,\mid\,\forall\, f\in \Bar{\Oss}_q\,:\,g\cdot f\in\Oss_q\subset \Bar{\Oss}_q\} 
\end{equation}
for all $q\in X'$ is called the conductor of $\Bar{\Oss}_{X'}$ in $\Oss_{X'}$. Here, we identify the meromorphic functions $\mathcal{M}_{X'}$ of $X'$ with the meromorphic functions of $\mathcal{M}_X$ by the map $f\mapsto f\circ\pi=\pi^\ast f$ which induces an isomorphism of sheaves $\mathcal{M}_{X'}=\pi_\ast\mathcal{M}_X$. The conductor \,$\mathfrak{c}$\, is a subsheaf of \,$\Oss_{X'}$\, because of \,$1\in \bar{\Oss}_{X'}$\,. The subvariety defined by this subsheaf of ideals is the set $S'$ of non-normal points in $X'$. Here, $S'$ equals the set of singular points in $X'$ because $X'$ is one-dimensional. With $S:=\pi^{-1}[S']$, the one-sheeted covering $\pi:X\setminus S\to X'\setminus S'$ is biholomorphic. Moreover, let $r_q$ be the radical
\begin{equation}\label{eq:radical}
r_q:=\{f\in \Bar{\Oss}_q\,\mid\,\forall\, p\in \pi^{-1}[\{q\}]\,:\,\pi^\ast  f(p)=0\}.
\end{equation}
\begin{Proposition}\label{prop1}
\begin{itemize}
\item[(a)]
$\delta_q:=\dim(\Bar{\Oss}_q/\Oss_q)<\infty$.
\item[(b)]
For $q\in X'$, there exists $n\in\mathbb{N}$ with
\begin{equation}\label{eq:inclusions}
\C+r_q^n\subset\C+\mathfrak{c}_q\subset\Oss_q\subset\C+r_q\subset\Bar{\Oss}_q.
\end{equation}
\end{itemize}
\end{Proposition}

\begin{proof}
(a) Due to \cite[Theorem~6.3.7]{DJo} $\Bar{\Oss}_q$ is a coherent sheaf on $X'$. This implies that $\Bar{\Oss}_q/\Oss_q$ is also a coherent sheaf (\cite[6.2.1]{DJo}). The support of a coherent sheaf $\Sss$ over $X'$ is defined as
\begin{equation}
\label{eq:support-sheaf}
\supp(\Sss):=\{q\in X'\mid\Sss_q\ne0\}.
\end{equation}
In particular, the support of $\Bar{\Oss}_q/\Oss_q$ is equal to $S'$ and has Weierstra{\ss} dimension $0$ (\cite[Definition~4.13]{DJo}).
We claim that the support of a coherent sheaf $\Sss$ on $X'$ is the subvariety
\[\{q\in X'\,\mid\, \forall \,f\in \Ann_q(\Sss)\,:\,f(q)=0\}\]
which is defined by the sheaf $\Ann(\Sss)$ of ideals called the annihilator of $\Sss$:
\[\Ann_q(\Sss):=\{f\in\Oss_q\,\mid\,\forall\, s\in\Sss_q\,:\, f\cdot s=0\in\Sss_q\}\quad\forall\, q\in X'.\]
In fact, if $q\in X'\setminus\supp(\Sss)$, then $\Sss_q=0$ and so $\Ann_q(\Sss)=\Oss_q$ contains functions not vanishing at $q$. Conversely, if $\Ann_q(\Sss)$ contains units of the local ring $\Oss_q$ with the maximal ideal $\{f\in\Oss_q\mid f(q)=0\}$, then $\Ann_q(\Sss)=\Oss_q$ and $f\cdot s=0$ for all $f\in \Oss_q$ and $s\in\Sss_q$. Since $\Sss_q$ is finitely generated (over $\Oss_q$), this implies $\Sss_q=0$ and proves the claim.

Clearly $\Sss$ is an $\Oss_{X'}/\Ann(\Sss)$-module. Due to the Noether normalisation (\cite[Corollary~3.3.19 with $k=0$ and $I=\Ann_q(\Sss)$]{DJo}), the stalks of a coherent sheaf $\Sss$ with zero-dimensional support are finitely generated $\C$-modules and therefore finite-dimensional vector spaces. In particular, $\delta_q$ is finite.

\noindent (b) We show $\C+r_q^n\subset \C+\Bar{\Oss}_q$ by decomposing
\begin{equation}\label{eq:decomposition}
\Bar{\Oss}_q=\bigoplus_{p\in\pi^{-1}[\{q\}]}\Oss_p,
\end{equation}
where the $\Oss_p$ are local rings with maximal ideals $r_p$ and
\[r_q=\bigoplus_{p\in\pi^{-1}[\{q\}]}r_p\;.\]
On the $\Bar{\Oss}_q$-module $\Bar{\Oss}_q/\Oss_q$, the natural homomorphism $\Oss_p\hookrightarrow\Bar{\Oss}_q$ induces for each $p\in\pi^{-1}[\{q\}]$ the structure of an $\Oss_p$-module. Since $\delta_q<\infty$, this module is finitely generated. The application of Krull's intersection Theorem \cite[Corollary~1.3.5]{DJo} to this module over the Noetherian local ring $\Oss_p$ with maximal ideal $r_p$ yields $n_p\in\mathbb{N}$ with $r_p^{n_p}\subset\mathfrak{c}_q\subset\Bar{\Oss}_q$. Choosing $n:=\max\{n_p\mid p\in\pi^{-1}[\{q\}]\}$, we obtain
\[\hspace{20mm}r_q^n=\left(\bigoplus\nolimits_{p\in\pi^{-1}[\{q\}]}r_p\right)^n=\bigoplus\nolimits_{p\in\pi^{-1}[\{q\}]}r_p^n\subset\mathfrak{c}_q\;.\]
\end{proof}
Finally, we remark that $\delta_q>0$ is equivalent to $q\in S'$. In fact, if $q\in S'$, then $\Bar{\Oss}_q\neq\Oss_q$, so $\delta_q>0$. Conversely, if $\delta_q=0$, then $\Bar{\Oss}_q=\Oss_q$, so $q\not\in S'$.

\subsection[Construction of a one-dim complex]{Construction of a one-dimensional complex analytic space from its normalisation}
In the preceding section, $X'$ was given. But now, a Riemann surface $X$ is given and we construct the one-dimensional singular complex analytic spaces $X'$ with normalisation $X$. We call such $X'$ singular curves. We will see that every singular curve $X'$ is given by the data $(X,S,R,\Oss_{X'})$. Here, $X$ is a Riemann surface, $S$ is a  discrete subset of $X$ and $R$ is an equivalence relation on $S$. This triple $(X,S,R)$ defines the topological space $X'$:

Extend $R$ to an equivalence relation on $X$ such that every $p\in X\setminus S$ is equivalent only to itself. Then, define the topological quotient space $X':=X/R$ with canonical map $\pi:X\to X'$. This map will turn out to be the normalisation map of $X'$ and the set $S':=\pi[S]\subset X'$ will be the set of singular points of $X'$. In particular, the equivalence class of a singular point $q\in S'$ contains all $p\in S$ which are in the preimage $\pi^{-1}[\{q\}]$ in the normalisation.

Let $\Oss_X$ be the sheaf of holomorphic functions on $X$ and $\Bar{\Oss}_{X'}:=\pi_\ast\Oss_X$ its direct image on $X'$. Then, the last datum $\Oss_{X'}$ is a sheaf of subrings of $\Bar{\Oss}_{X'}$ with two properties. At first, $\Oss_q\subsetneq\Bar{\Oss}_q$ if and only if $q\in S'$. Secondly, for every $q\in S'$, there exists $n\in \mathbb{N}$ such that
\begin{equation}
\label{eq:datainclusion}
\C+r_q^n\subset\Oss_q \subset\C+r_q \;,
\end{equation}
 where $r_q$ is the radical of $\Bar{\Oss}_q$ as defined in~\eqref{eq:radical}.
\begin{Remark}\label{rem:data}\begin{enumerate}
\item For $q\in S'$, there exists $n\in\mathbb{N}$ such that \ $\C+r_q^n\subset\Oss_q$ if and only if $\delta_q=\dim(\Bar{\Oss}_q/\Oss_q)<\infty$. In fact, one has
\[
\dim\big(\Bar{\Oss}_q/(\C+r_q^n)\big)=n\cdot\#\pi^{-1}[\{q\}]-1.
\]
The other direction follows from the proof of Proposition~\ref{prop1}~(b).
\item\label{separate} The condition \,$\Oss_q \subset \C+ r_q$\, has the consequence that the holomorphic functions on \,$X'$\, do not separate the points of $\pi^{-1}[\{q\}]$.
\item Choosing the sheaf $\Oss_{X'}$ is equivalent to choosing subrings $\Oss_q$ of $\Bar{\Oss}_q$ obeying~\eqref{eq:datainclusion} for all $q\in S'$. Indeed, together with $\Oss_q=\Bar{\Oss}_q$ for $q\in X'\setminus S'$, any such choice defines a sheaf $\Oss_{X'}$ on $X'$ since restricting the stalks $\Oss_q$  for $ q\in S'$ to a subset of finite codimension does not affect the subrings $\Oss_q$ for $q\not\in S'$.
\end{enumerate}
\end{Remark}
\begin{Proposition}\label{P:covers}
Let $(X,S,R,\Oss_{X'})$ be given as above. Then, up to isomorphism, there exists exactly one singular curve $X'$ with normali- sation $X$, set of singular points $S'=S/R$ and sheaf of holomorphic functions $\Oss_{X'}$. Moreover, every singular curve $X'$ can be obtained in this way.
\end{Proposition}
\begin{proof}
As before, we extend the relation $R$ to $X$ and consider $X':=X/R$ as topological space with the canonical projection $\pi:X\to X'$. We view $X'$ as a ringed space with the sheaf $\Oss_{X'}$. Then $\pi|_{X\setminus S}:X\setminus S\to X'\setminus S'$ induces an isomorphism of ringed spaces via the direct image of sheaves. This implies that $X'\setminus S'$ is a smooth complex curve and that $\pi|_{X\setminus S}$ is biholomorphic. It remains to show that $X'$ is a complex curve with normalisation $X$.

As a first step, we prove that for $n\in\mathbb{N}$ and $q\in S'$, the integral closure of $\C+r_q^n$ is equal to $\Bar{\Oss}_q$. For every $f\in r_q$, we have $f^n\in r_q^n$, so $r_q$ is integral over $r_q^n$. Moreover, every $f\in \Bar{\Oss}_q$ which is locally constant on $X$ is a root of a polynomial with constant coefficients and therefore integral over $\C$. Due to~\cite[Example~4.4.7~(1)]{DJo}, $\Bar{\Oss}_q=(\C\{x\})^{\#\pi^{-1}[\{q\}]}$ is integrally closed in the stalk of meromorphic germs. It follows that $\Bar{\Oss}_q$ is the integral closure of $\C+r_q^n$ since every $f\in\Bar{\Oss}_q$ is the sum of an element of $r_q$ and a germ which is locally constant on $X$. This first step together with~\eqref{eq:datainclusion} yields that $\Bar{\Oss}_q$ is also the integral closure of $\Oss_q$.

In the second step, we transfer the application of \cite[III~\S 3.12 Lem\-ma~10]{Se} in the proof of \cite[IV~\S1.3 Proposition 2]{Se} and the proof of that Lemma to our situation ($A=\C+r_q$ and $B=\Oss_q$). Instead of showing that $\C+r_q$ is of finite type in the sense of Serre, we find $f_1,\ldots,f_k\in\C+r_q$ such that the canonical map $\varphi:\C\{f_1,\ldots,f_k\}\to \C+r_q$ is surjective. This means that $\C\{f_1,\ldots f_k\}/\kernel(\varphi)\simeq\C+r_q$. More precisely, let $\pi^{-1}[\{q\}]=\{p_1,\ldots,p_k\}$ and choose $f_i$ as a local coordinate at $p_i$ (i.e. $f_i(p_i)=0$ and $f_i'(p_i)\neq 0$) and as identically zero near $p_j$ for $j\ne i$. Due to the first step, each $f_i$ is a root of a monic polynomial $P_i$ with coefficients in $\C+r_q^n\subset\Oss_q$. We collect the coefficients of all these polynomials, say $b_1,\ldots,b_r\in\Oss_q$. Let $C$ be the image of the canonical map $\psi:\C\{b_1,\ldots,b_r\}\to\Oss_q$. Again, $C$ is isomorphic to $\C\{b_1,\ldots,b_r\}/\kernel(\psi)$ and therefore Noetherian~\cite[Corollary~1.1.4]{DJo}.

We claim that $\C+r_q$ is a Noetherian $C$-module. Due to the Weierstra{\ss} Division Theorem \cite[Theorem~3.2.3]{DJo}, $C\{f_1\}/(P_1)$ is as a $C$-module generated by $f_1,f_1^2,\ldots,f_1^{\deg P_1-1}$. Inductively, this shows that $C\{f_1,\ldots,f_k\}/(P_1,\ldots P_k)$ is a finitely generated $C$-module. Because $C$ is Noetherian, the claim follows from \cite[Lemma~1.2.15]{DJo}.

By \cite[Definition~1.2.13]{DJo}, the submodule $\Oss_q$ of $\C+r_q$ is a finitely generated $C$-module with generators $y_1,\ldots,y_m\in \Oss_q$. Then, the canonical map $\C\{b_1,\ldots ,b_r,y_r,\ldots y_m\}\to\Oss_q$ with kernel $I$ is surjective. Hence, $\Oss_q\simeq\C\{b_1,\ldots,b_r,y_1,\ldots ,y_m\}/I$. The ideal $I$ of the Noetherian ring $\C\{b_1,\ldots ,b_r,y_r,\ldots y_m\}$ is finitely generated. In particular, $X'$ nearby $q$ is a complex analytic space. This holds for all $q\in S'$ and therefore $X'$ is a complex curve. By \cite[Definition~1.5.3]{DJo}, the first step shows that $\Bar{\Oss}_q$ is the normalisation of the local ring $\Oss_q$. Since this holds for all $q\in S'$, $X$ is the normalisation of $X'$ by \cite[Definition~4.4.5]{DJo}.

Conversely, if a singular curve $X'$ is given, then let $X$ be the normalisation of $X'$ and $S,R$ and $\Oss_{X'}$ as in Section~\ref{sec:1}. The data $(X,S,R,\Oss_{X'})$ yields the singular curve $X'$. More precisely, $(X,S,R)$ uniquely determines the topological space $X'$ and $\Oss_{X'}$ is the structure sheaf of the singular curve $X'$.
\end{proof}
\begin{Remark}
Compared to the situation in \cite[IV~\S1.3 Proposition~2]{Se}, we have the advantage that $X'$ is already determined by $X$ at regular points. So we only have to describe $X$ nearby the points of $S'$. At a first sight, choosing $A=\Bar{\Oss}_q$ might seem to be more natural than $A=\C+r_q$ in the proof of the proposition. However, $\Bar{\Oss}_q$ has several maximal ideals. With $A=\C+r_q$, both rings $A$ and $B$ are local rings of the form $\C\{x_1,\ldots,x_k\}/I$ with suitable ideals $I$. Establishing this property of $B$ is the essential step in the proof of Proposition~\ref{P:covers}.
\end{Remark}
\begin{Example}\label{example1} Examples of singular points on a Riemann surface $X$.
\begin{enumerate}
\item\textbf{Ordinary double point}\label{double point}. Let $p_1\neq p_2\in X$, $S=\{p_1,p_2\}$ and $R$ the relation which identifies $p_1$ with $p_2$. At the unique point $q\in S/R$, the ring $\Oss_q$ is the subring of $\Bar{\Oss}_q=\Oss_{p_1}\oplus\Oss_{p_2}$ of elements $f_1\oplus f_2$ with $f_1(p_1)=f_2(p_2)$. In this case $\Oss_q$ equals $\C+r_q$ and is isomorphic to $\C\{x_1,x_2\}/(x_1\cdot x_2)$. Here $\delta_q=1$.
\item\textbf{Ordinary cusp}.\label{cusp} Let $p\in X$, $S=\{p\}=\{q\}=S'$. At $q$, the ring $\Oss_q$ is the subring of $\Bar{\Oss}_q=\Oss_p$ of elements $f$ with $f'(p)=0$. In this case, $\Oss_q$ equals $\C+r_q^2$. Hence, $\Bar{\Oss}_q\simeq\C\{t\}$  and $\Oss_q$ is the subring generated by $x_1=t^2$ and $x_2=t^3$. The kernel of the natural map $\C\{x_1,x_2\}\to\C\{t\}$ is defined by $(x_2^2-x_1^3)$, and therefore  $\Oss_q$ is isomorphic to $\C\{x_1,x_2\}/(x_2^2-x_1^3)$. Here, $\delta_q=1$.
\item\textbf{Singular point defined by a divisor}. Let $D=\sum_{i=1}^Nn_ip_i$ be a finite divisor on $X$ with $n_i\in\mathbb{N}\setminus\{0\}$, $S=\{p_1,\ldots,p_N\}$ and $R$ the unique relation identifying all points of $S$ to one point $q\in S'$. At $q$, the ring $\Oss_q$ is the subring $\C+\bigoplus_{i=1}^N(r_{p_i})^{n_i}$ of $\Bar{\Oss}_q=\bigoplus_{i=1}^N\Oss_{p_i}$. Here, $\delta_q=\deg D-1$. For $D=p_1+p_2$, we obtain the ordinary double point and for $D=2p$ the ordinary cusp.
\end{enumerate}\end{Example}
\section{Generalised Divisors}\label{se:generalised divisors}
In this section, we transfer the concept of a line bundle on a Riemann surface to a singular curve $X'$ for which we use the notations introduced in Section~\ref{sec:1}.

There exists a notion of line bundles on singular curves, see \cite[Chapter~II~\S3]{GPR}. However, the concept we introduce here is based on divisors instead of on line bundles. On Riemann surfaces, both concepts are equivalent, see \cite[\S29, \S18]{Fo}, but on singular curves, generalised divisors are more general than line bundles. In Section~\ref{se:BA}, we will use generalised divisors to construct so-called Baker-Akhiezer functions on singular curves. Our aim is to construct Baker-Akhiezer functions on singular curves in the most general setting we can imagine. 


On smooth surfaces, divisors \,$D$\, correspond to finitely generated subsheaves \,$\Oss_D$\, of \,$\Mss$\,, compare \cite[\S16.4]{Fo}. This correspondence inspires the following definition:
\begin{Definition}\cite[\S 1]{Ha86}
A \emph{generalised divisor} on $X'$ is a finitely generated subsheaf $\Sss$ of the sheaf of meromorphic functions on $X'$.\\
The \emph{support} $\mathrm{supp}(\Sss)$ of a generalised divisor $\Sss$ is the set of all $q\in X'$ such that $\Sss_q\neq \Oss_{q}$.\\
\end{Definition}

Here, we extend the definition of the support for classical divisors to generalised divisors, in deviation from the definition of the support of a coherent sheaf \eqref{eq:support-sheaf}.

\begin{Remark}\label{R:invertible}
The sheaves corresponding to line bundles are characterized as locally free sheaves of rank \,$1$. These sheaves are also called invertible because the tensor product with such a sheaf has an inverse~\cite[p. 143]{Ha}. A subsheaf of the sheaf of meromorphic functions which is locally free cannot have rank \,$\geq 2$\, since a sheaf homomorphism of \,$\Oss_{X'}^n$\, into a generalised divisor cannot be injective if \,$n\geq 2$\,. We will see in Corollary~\ref{C:meromorphic-functions} that every locally free sheaf of rank \,$1$\, has a meromorphic section. Therefore, the correspondence between line bundles and divisors on Riemann surfaces (\cite[\S 29]{Fo}) yields a correspondence between line bundles and locally free generalised divisors on singular curves. 
In particular, invertible sheaves are the same as locally free generalised divisors.
\end{Remark}
\begin{Proposition}\label{prop3}
The support of a generalised divisor $\Sss$ is a discrete subset of $X'$.
\end{Proposition}
\begin{proof}
For \,$q\in X'$, there exists an open neighbourhood \,$U$\, of \,$q$\, such that \,$\Sss$\, is generated as an \,$\Oss_U$-module by finitely many meromorphic functions \,$f_1,\dotsc, f_n$\,, where \,$n$\, can depend on \,$q$. If all \,$f_i$\, have no pole and at least one \,$f_i$\, does not vanish at \,$q$, then \,$\Sss_q$\, equals \,$\Oss_q$\,.
Therefore, the support of $\Sss$ is a discrete subset of $X'$.
\end{proof}
Near a regular point of \,$X'$, a generator of maximal pole order respectively minimal zero order alone suffices to generate \,$\Sss$. In particular, on the regular set of \,$X'$\,, \,$\Sss$\, equals \,$\Oss_D$\, for some classical divisor \,$D$\,.  

The next step is to define the degree of $\Sss$. Note that for any pair \,$\Sss_1, \Sss_2$\, of generalised divisors, there exists a divisor \,$\Sss$\, with
\begin{equation}
\Sss_1\subset \Sss \quad \text{ and } \quad \Sss_2\subset \Sss \; . 
\label{eq:smallest-common-divisor}
\end{equation}
Indeed, the smallest \,$\Sss$\, with this property is generated locally by the union of local generators of \,$\Sss_1$\, and \,$\Sss_2$\,. 
\begin{Proposition}
For generalised divisors \,$\Sss_1, \Sss_2$\, and \,$\Sss$\, with finite support and \eqref{eq:smallest-common-divisor}, the difference
\[ \dim H^0(X',\Sss/\Sss_2) - \dim H^0(X',\Sss/\Sss_1) \]
does not depend on \,$\Sss$\,.
\end{Proposition}
\begin{proof}
Let \,$\Sss$\, and \,$\Sss'$\, be given with finite support and \,$\Sss_1,\Sss_2 \subset \Sss, \Sss'$\,. We have to show that
\begin{multline}
\label{eq:H0-to-show}
\dim H^0(X',\Sss/\Sss_2)-\dim H^0(X',\Sss/\Sss_1)=\\=\dim H^0(X',\Sss'/\Sss_2) - \dim H^0(X',\Sss'/\Sss_1) \; .
\end{multline}
There always exists \,$\Sss''$\, with finite support and \,$\Sss,\Sss'\subset \Sss''$\,. Hence, we may suppose without loss of generality that \,$\Sss \subset \Sss'$\, holds.

For \,$k=1,2$, the sequences
\[ 0 \longrightarrow  \Sss'/\Sss \hookrightarrow \Sss'/\Sss_k \longrightarrow \Sss/\Sss_k \longrightarrow 0 \]
are exact since $\Sss/\Sss_k$ is the quotient of $\Sss'/\Sss_k$ and $\Sss'/\Sss$. Hence, the corresponding long exact sequences on the cohomology spaces are also exact, see \cite[\S15]{Fo}. Because the generalised divisors involved have finite support, the arguments in the proof of \cite[16.7~Lemma]{Fo} show that the first cohomology groups of all sheaves in the long exact sequences vanish. Therefore, they reduce to
\[
0\longrightarrow H^0(X',\Sss'/\Sss)\longrightarrow H^0(X',\Sss'/\Sss_k) \longrightarrow H^0(X',\Sss/\Sss_k) \longrightarrow 0.
\]
Due to the exactness, the alternating sums of the dimensions vanish and thus, we have
\[
\dim H^0(X',\Sss'/\Sss_k)=\dim H^0(X',\Sss'/\Sss) + \dim H^0(X',\Sss/\Sss_k).
\]
Taking the difference of these equations for \,$k=1$\, and \,$k=2$\, gives \eqref{eq:H0-to-show}.
\end{proof}
\begin{Definition}
The \emph{degree} of $\Sss$ with finite support is defined as 
\[
\deg(\Sss):= \dim H^0(X',\Sss'/\Oss_{X'}) - \dim H^0(X',\Sss'/\Sss), 
\]
where \,$\Sss'$\, is any generalised divisor with finite support containing \,$\Sss$\, and \,$\Oss_{X'}$\,. 
\end{Definition}
Because the support of a given generalised divisor \,$\Sss$\, is discrete by Proposition~\ref{prop3}, there exists for each \,$q\in \supp(\Sss)$\, a generalised divisor \,$\Sss(q)$\, with \,$\supp(\Sss(q))=\{q\}$\, and \,$\Sss(q)_q = \Sss_q$\,. We call \,$\deg_q(\Sss) := \deg(\Sss(q))$\, the \emph{degree of \,$\Sss$\, at \,$q$\,}. The \emph{classical divisor \,$D(\Sss)$\, associated to \,$\Sss$\,} is defined by 
\[ D(\Sss) := \sum_{q\in \supp(\Sss)} \deg_q(\Sss) \cdot q \; . \]
Vice versa, for every classical divisor \,$D$\, on \,$X'$, there exists a generalised divisor \,$\Sss$\, with \,$D(\Sss)=D$\,. Indeed, \,$\Sss$\, can be chosen as the locally free divisor generated locally by a suitable meromorphic function \,$f$\,. We choose \,$f$\, such that the classical divisor on \,$X$\, of the local generator \,$f$, considered as a meromorphic function on \,$X$, is projected onto \,$D(\Sss)$\, by the normalisation map \,$\pi$\,. Then \,$\deg_q(\Sss)=\deg_q(D)$\, holds for every \,$q\in X'$\, since
\begin{equation}
\label{eq:deg-deg}
\sum_{p\in \pi^{-1}[\{q\}]} \deg_{p}(f\cdot \Oss_X) =\deg_q(f\cdot \Bar{\Oss}_{X'}) - \delta_q = \deg_q(f\cdot \Oss_{X'}) \; . 
\end{equation}
Note that on the left-hand side, the degree is calculated in \,$X$\, whereas in the middle and on the right-hand side, it is calculated in \,$X'$\,. However, \,$\Sss$\, is not determined uniquely by \,$D(\Sss)$\,, because the projection of divisors on \,$X$\, onto \,$X'$\, is not injective if \,$\pi$\, is not injective. On smooth complex curves \,$X'$, there is a 1--1 correspondence between generalised divisors \,$\Sss$\, and classical divisors \,$D(\Sss)$\,. 

Furthermore, if \,$X'$\, is not smooth, there exist generalised divisors on \,$X'$\, which are not locally free and therefore, a generalised divisor on \,$X'$\, is then not uniquely determined by its classical divisor.
\begin{Example}[Ordinary Cusp]
\label{E:cusp}
Consider the curve $X'=\{(x_1,x_2)\in \C^2\,|\,x_2^2=x_1^3\}$ which is singular at $(0,0)$ with the normalisation \,$\pi: \C \to X',\; t \mapsto (t^2,t^3)$, see Example~\ref{example1}~\eqref{cusp}. We choose the generalised divisor $\Sss:=\Bar{\Oss}_{X'}$. Then, $\Sss$ is generated by $1$ and $\frac{x_2}{x_1}$. Therefore, we have
\[
\deg(\Sss)=\dim H^0(X',\Sss/\Oss_{X'})=1=\delta_{(0,0)} \; . 
\]
The classical divisor associated to \,$\Sss$\, is \,$D(\Sss)=(0,0)$\,. In this case, \,$\pi$\, is injective, so there exists a unique locally free generalised divisor \,$\Sss'$\, associated to \,$D(\Sss)$\,. \,$\Sss'$\, is generated by \,$\tfrac{x_1}{x_2}$\,. Thus, we have \,$\Sss' \neq \Sss$\, and therefore, \,$\Sss$\, is not locally free. 
\end{Example}
\section{The $\Sss'$-halfway normalisation}\label{se:halfway}
Only locally free generalised divisors determine the structure sheaf of the underlying singular curve uniquely. In Example~\ref{E:cusp}, the first generalised divisor on the singular curve with cusp is also a generalised divisor on the normalisation of that curve. In this section we investigate possible choices of structure sheaves for a given generalised divisor. Here, we are interested only in extensions of the structure sheaf of the original singular curve. More precisely: For a given singular curve \,$X'$\, with a generalised divisor \,$\Sss'$\, on \,$X'$, we look for another pair \,$(X'',\Sss'')$\, with a holomorphic map \,$\pi': X'' \to X'$\, such that \,$\pi'_\ast\Sss''=\Sss'$\,. Because of the latter condition, the map \,$\pi'$\, has to be birational for non-trivial \,$\Sss'$\,. Therefore, we consider a one-sheeted covering \,$\pi': X'' \to X'$\,, i.e.~\,$\pi'$\, is regular and biholomorphic away from a discrete subset of $X''$.  This set is contained in the preimage of the singular set \,$S'$\,. 
For any generalised divisor $\Sss''$ on $X''$, the direct image $\pi'_\ast\Sss''$ is a generalised divisor on $X'$. 
\begin{Lemma}
\label{L:projectdivisor}
On a one-sheeted covering \,$\pi': X'' \to X'$, there exists a generalised divisor \,$\Sss''$\, on \,$X''$\, with \,$\pi'_\ast\Sss''=\Sss'$\, if and only if the multiplication with meromorphic functions in \,$\pi'_\ast\Oss_{X''}$\, maps \,$\Sss'$\, into itself. Such one-sheeted coverings $\pi':X''\to X'$ are in 1--1 correspondence with sheaves $\Rss$ of subrings of $\Bar{\Oss}_{X'}$ which act on $\Sss'$ and contain $\Oss_{X'}$.
\end{Lemma}
\begin{proof}
Let $\pi':X''\to X'$ be a one-sheeted covering and \,$\Sss''$\, the \,$\Oss_{X''}$-module locally generated by the pullbacks of some choice of local gene- rators of \,$\Sss'$\,. Because \,$\pi_\ast\Oss_{X''}$\, contains \,$\Oss_{X'}$\,, \,$\pi'_\ast\Sss''$\, is the \,$\pi'_\ast\Oss_{X''}$-module generated by these generators of \,$\Sss'$\,. If \,$\Sss'$\, is a \,$\pi'_\ast\Oss_{X''}$-module, then \,$\pi_\ast\Sss''=\Sss'$\, holds. Conversely, if \,$\pi'_\ast\Sss''=\Sss'$\, holds, then \,$\Sss'$\, is a \,$\pi'_\ast\Oss_{X''}$-module since \,$\Sss''$\, is a \,$\Oss_{X''}$-module.

For the proof of the 1--1 correspondence, we first note that for such one-sheeted coverings, $\pi'_\ast\Oss_{X''}$ is a sheaf of subrings of $\Bar{\Oss}_{X'}$ which acts on $\Sss'$ and contains $\Oss_{X'}$. Conversely, let \,$\Rss$\, be a sheaf of subrings of \,$\Bar{\Oss}_{X'}$\, which acts on $\Sss'$ and contains $\Oss_{X'}$. For \,$q\in S'$, the stalks \,$\Rss_q$\, are not necessarily contained in \,$\C+r_q$\,, see Equation~\eqref{eq:datainclusion}. \,$\Rss$\, implicitly determines the topological space \,$X''$\,: For each \,$q \in S'$, only those points of \,$\pi^{-1}[\{q\}]$\, correspond to different points in \,$X''$\, which are separated by elements of $\Rss_q$. Define an equivalence relation \,$R''$\, on \,$S$\, such that two points of \,$S$\, are equivalent if and only if each element of \,$\Rss$\, (considered as a function on \,$X$) takes the same value at the two points. The triple \,$(X,S,R'')$\, defines \,$X''$\, as a topological space as in Proposition~\ref{P:covers} and also defines the map \,$X\to X''$. Because of the inclusion \,$\Oss_{X'} \subset \Rss$\,, the holomorphic functions on \,$X'$, considered as holomorphic functions on \,$X$, are constant on the equivalence classes of \,$R''$. Hence, the equivalence classes of \,$R$\, are subsets of the equivalence classes of \,$R''$\,. This defines the one-sheeted covering $\pi':X''\to X'$ which splits $\pi$ into \,$X \to X'' \to X'$\,.

By definition of $R''$, there exists for every \,$p'' \in \pi'^{-1}[\{q\}]$\, an element \,$P_{p''} \in \Rss_q$\, which is identically \,$1$\, near all points in the equivalence class of \,$R''$\, corresponding to \,$p''$\, and identically \,$0$\, near all other points of \,$\pi^{-1}[\{q\}]$\,. Because of \,$\sum_{p''\in \pi'^{-1}[\{q\}]} P_{p''}=1$\,, the decomposition~\eqref{eq:decomposition} induces the following decomposition of $\Rss_q$ into subrings, see~Remark~\ref{rem:data}~\eqref{separate}:
$$ \Rss_q = \bigoplus_{p''\in\pi'^{-1}[\{q\}]} P_{p''}\cdot \Rss_q \; . $$
Therefore, there exists a unique subsheaf \,$\Oss_{X''}$\, of the direct image of \,$\Oss_X$\, on \,$X''$\, with $\pi'_\ast\Oss_{X''}=\Rss$\,. For all equivalence classes \,$q''$\, of $R''$, the ring \,$\Oss_{X'',q''}$\, is contained in \,$\C + r_{q''}$\, with the radical \,$r_{q''}$\,  defined analogously as \,$r_q$\, in Equation~\eqref{eq:radical}. Due to Proposition~\ref{P:covers}, the data $(X,S,R'',\Oss_{X''})$ corresponds to a singular curve $X''$ with normalisation $X$ and a map $\pi':X''\to X'$. Since $\pi'_\ast\Oss_{X''}=\Rss$ contains $\Oss_{X'}$ with equality on $X'\setminus S'$, the map $\pi'$ is a one-sheeted covering and $\pi'_\ast\Oss_{X''}$ acts on $\Sss'$.
\end{proof}
\begin{Definition}
\label{D:middleding}
For a generalised divisor \,$\Sss'$\, on a singular curve \,$X'$, let \,$\pi_{X(\Sss')}: X(\Sss') \to X'$\, be the unique one-sheeted covering such that 
\[ (\pi_{X(\Sss')})_\ast \Oss_{X(\Sss')} = \{ f \in \Bar{\Oss}_{X'} \,|\, f\cdot g \in \Sss' \text{ for all \,$g\in \Sss'$} \} \; . \]
We call \,$X(\Sss')$\, the \emph{\,$\Sss'$-halfway normalisation} of \,$X'$\,.
\end{Definition}
By definition, \,$(\pi_{X(\Sss')})_\ast \Oss_{X(\Sss')}$\, is the largest subring of \,$\Bar{\Oss}_{X'}$\, that acts on \,$\Sss'$\,. A two-fold application of the Lemma~\ref{L:projectdivisor} to the pairs $(X',\Sss')$ and $(X'',\Sss'')$ shows that the one-sheeted coverings described in this lemma are exactly those which split the normalisation map $\pi$ into
$$X \to X(\Sss')\to X''\to X'.$$
The name ``$\Sss'$-halfway normalisation'' should remind the reader that $X(\Sss')$ is a one-sheeted covering in between the normalisation $X$ and $X'$.

In Example~\eqref{E:cusp}, the \,$\Sss$-halfway normalisation for the generalised divisor $\Sss$ is equal to the usual normalisation $X(\Sss)=X$ of \,$X'$\,, and the \,$\Sss'$-halfway normalisation for the locally free generalised divisor $\Sss'$ is $X(\Sss')=X'$. Both corresponding generalised divisors on the respective halfway normalisations constructed in Lemma~\ref{L:projectdivisor} are locally free.

Because locally free generalised divisors are better-behaved (see Remark~\ref{R:invertible}), given a pair \,$(X',\Sss')$, it is natural to look for a pair \,$(X'',\Sss'')$\, as in Lemma~ \ref{L:projectdivisor} such that \,$\Sss''$\, is locally free. The following lemma gives a necessary condition for \,$X''$\,.
 
\begin{Lemma}
\label{L:middleding-condition}
In the situation of Lemma~\ref{L:projectdivisor}, the generalised divisor \,$\Sss''$\, can be locally free only if \,$X''=X(\Sss')$\, holds. 
\end{Lemma}

\begin{proof}
Let \,$\pi': X'' \to X'$\, be a one-sheeted covering with a locally free generalised divisor \,$\Sss''$\, such that \,$\pi'_* \Sss'' = \Sss'$\, holds. 
Locally at \,$q\in X'$\,, there exists an $f\in \Sss'_{q}$ such that $(\pi'_*\Oss_{X''})_q \to \Sss'_q$, $g\mapsto gf$ is an isomorphism. Then, the ring $(\pi'_*\Oss_{X''})_q$ is equal to 
$$
\{ g \in \Bar{\Oss}_{X',q} \,|\, g f \in \Sss'_q \}
= 
\{ g \in \Bar{\Oss}_{X',q} \,|\, g h \in \Sss'_q \text{ for all \,$h\in \Sss'_q$} \}
 \; 
$$
since the multiplication with $f$ is an injective map from the stalk of meromorphic functions into itself.
This implies \,$(\pi_{X(\Sss')})_\ast \Oss_{X(\Sss')} =\pi'_*\Oss_{X''}$\, by Definition~\ref{D:middleding} of \,$X(\Sss')$\,. Now, Lemma~\ref{L:projectdivisor} yields that \,$X(\Sss')=X''$\,.
\end{proof}
\begin{Theorem}\label{two sheeted}
Let \,$X'$\, be a (possibly branched) two-sheeted covering above a smooth Riemann surface \,$Y$\,. 
Then, for every generalised divisor $\Sss'$ on \,$X'$, the unique sheaf \,$\Sss''$\, on the \,$\Sss'$-halfway normalisation \,$X(\Sss')$\, in Lemma~\ref{L:projectdivisor} is locally free.
\end{Theorem}
By definition, every hyperelliptic singular curve \,$X'$\, is a branched two-sheeted covering over \,$\P$\,.
\begin{proof}
On the complement of the branch points of the two-sheeted covering \,$X' \to Y$\, and the singularities of \,$X'$, there exists a unique holomorphic involution \,$\iota$\, that interchanges any two points with the same image in \,$Y$\,. By the Riemann Extension Theorem (\cite[Theorem~3.1.15]{DJo}), \,$\iota$\, extends to an involution on the normalisation \,$X$\,. Let \,$q\in X'$\, be either a singular point or a branch point. The involution \,$\iota$\, acts on \,$\bar{\Oss}_{X',q}$\,. We decompose \,$\bar{\Oss}_{X',q}$\, into the symmetric part \,$\bar{\Oss}_{X',q}^+$\, and the anti-symmetric part $\bar{\Oss}_{X',q}^-$ with respect to \,$\iota$\,. For $f\in \Oss_{X',q}$, we write \,$f=f_++f_-$\, with \,$f_\pm \in \bar{\Oss}^\pm_{X',q}$\,. Due to \cite[Proposition~8.2]{Fo}, the functions in \,$\bar{\Oss}_{X',q}^+$\, are pullbacks of holomorphic functions on \,$Y$. Hence, \,$\bar{\Oss}_{X',q}^+ = \mathbb{C}\{y\} \subset \Oss_{X',q}$, where \,$y$\, is the pullback of a local coordinate on \,$Y$\, centered at the image of \,$q$\, in $Y$. 
In particular, $f_+\in \Oss_{X',q}$ and $f_-=f-f_+\in \Oss_{X',q}$ for $f\in \Oss_{X',q}$. Consequently, the involution $\iota$ acts on \,$\Oss_{X'}$\, and thus on $X'$. 

  The elements of \,$\bar{\Oss}_{X',q}^- \cap \Oss_{X',q}$\, vanish at \,$q$\,: If \,$\pi^{-1}[\{q\}]$\, contains only one point, then all elements of \,$\bar{\Oss}_{X',q}^-$\, vanish at \,$q$\,. If \,$\pi^{-1}[\{q\}]$\, contains two points interchanged by \,$\iota$\,, then the elements of \,$\bar{\Oss}_{X',q}^- \cap \Oss_{X',q}$\, vanish since they have the same and simultaneously the opposite value at these two points. 

Since $\delta_q<\infty$, the codimension of $\bar{\Oss}_{X',q}^- \cap \Oss_{X',q}$\,in \,$\bar{\Oss}_{X',q}^-$\, is finite. So for all \,$p\in\pi^{-1}[\{q\}]$, there exists an element of \,$\bar{\Oss}_{X',q}^- \cap \Oss_{X',q}$\, which vanishes at \,$p$\, of minimal order. Let \,$x$\, be the sum of two such elements if \,$\pi^{-1}[\{q\}]$\, contains two points. Thereby, we obtain \,$x \in \bar{\Oss}_{X',q}^- \cap \Oss_{X',q}$\, which vanishes at all \,$p\in \pi^{-1}[\{q\}]$\, of minimal order. Then for every element \,$\widetilde{x}$\, of \,$\bar{\Oss}_{X',q}^-$, the quotient \,$\widetilde{x}/x$\, is locally bounded and thus belongs to \,$\bar\Oss_{X',q}^+$\,.
Therefore, \,$\bar{\Oss}_{X',q}^- \cap \Oss_{X',q}$\, is a free \,$\bar{\Oss}_{X',q}^+$-module generated by $x$. 
Since $x^2\in \bar\Oss^+_{X',q}=\C\{y\}$, there exists an \,$n\in \mathbb{N}$\, and a unit \,$u \in \mathbb{C}\{y\}$\, with \,$x^2=y^n\cdot u$\,. The unit \,$u$\, has a square root in \,$\C\{y\}$\, and by multiplying \,$x$\, with $\sqrt{u}^{-1}$ the equation simplifies to
$$ x^2=y^n \; . $$
We only need to consider the singularities of $X'$ because everywhere else, $S'$ is locally free. Therefore, we may suppose $n\geq 2$. The cases of even $n$ and of odd $n$ are treated separately.

For $n=2m$ even, the preimage of $q$ in $X$ consists of two points. Since $\big(\frac{x}{y^m}\big)^2=1$,  $\frac{x}{y^m}=\pm 1$ generates \,$\bar{\Oss}_{X',q}$\, as a \,$\bar{\Oss}_{X',q}^+(=\C\{y\})$-module, so the mapping $\C\big\{y,\frac{x}{y^m}\big\}\to \bar{\Oss}_{X',q}$ is surjective. 

For $n=2m+1$ odd, the preimage of $q$ in $X$ consists of only one point. Then, $\bar{\Oss}_{X',q}=\C\big\{\frac{x}{y^m}\big\}$.

For every \,$p\in \pi^{-1}[\{q\}]$, let \,$f_p$\, minimize \,$\mathrm{ord}_p(\pi^*f_p)$\, in \,$\Sss_q'$\,. Then a suitable linear combination \,$f$\, of the (one or two) \,$f_p$\, minimizes \,$\sum_{p\in \pi^{-1}[\{q\}]} \mathrm{ord}_p(\pi^*f)$\, in \,$\Sss_q'$\,. In this situation, we have 
$$ f \Oss_{X',q} \subset \Sss_q' \subset f \bar{\Oss}_{X',q} \; . $$
We decompose \,$\Sss_q'$\, in \,$f \bar{\Oss}_{X',q} = f \Oss_{X',q} \oplus (f \bar{\Oss}_{X',q} / f \Oss_{X',q})$\,. Then, $\Sss_q'$ is uniquely determined by the image of $\Sss_q'$ in $f \bar\Oss_{X',q}/ f \Oss_{X',q}$.

Clearly, one has 
\[
\dim H^0(X',f\bar{\Oss}_{X',q} /f \Oss_{X',q})= \dim H^0(X',\bar{\Oss}_{X',q}/\Oss_{X',q}) = \delta_q =m \; . 
\]
with a basis of $H^0(X',\bar{\Oss}_{X',q}/\Oss_{X',q})$  given by
\[
\frac{x}{y}, \, \frac{x}{y^2}, \, \dots , \frac{x}{y^m}\;,
\]
regardless of whether \,$n$\, is even or odd.
Because of \,$\Sss_q' \subset f \bar{\Oss}_{X',q}$, there exists a greatest integer $\ell$ such that $\frac{x}{y^\ell}\cdot f\in \Sss'_q$. Then, 
	\[
	\frac{x}{y^\ell}y^k f = \frac{x}{y^{\ell-k}}f\in \Sss_q'
	\]
	and so, 
		\[
		\Sss'_q=\Oss_{X',q}f+\C\frac{x}{y^\ell}f+\C\frac{x}{y^{\ell-1}}f +\dots + \C\frac{x}{y}f \; . 
		\]
Therefore, $\Sss'_q$ is locally free with generator \,$f$\, with respect to the image of \,$\C\big\{y,\tfrac{x}{y^\ell}\big\} \to \bar{\Oss}_{X',q}$\,. The kernel of this map is generated by 
	\begin{equation*} 
		\left(\frac{x}{y^\ell}\right)^2-y^{n-2\ell} \;.
	\end{equation*}
	Hence, the $\Sss'$-halfway normalisation \,$X(\Sss')$\, is the one-sheeted covering \,$\pi':X''\to X'$\, from Lemma~\ref{L:projectdivisor} such that \,$\pi'_* \Oss_{X''}$\, is equal to the image of \,$\C\big\{y,\tfrac{x}{y^\ell}\big\} \to \bar{\Oss}_{X',q}$\, and \,$\Sss'' = f \cdot \Oss_{X''}$\, is locally free on \,$X''$\,. 
\end{proof}

\subsection{The $\Sss'$-halfway normalisation of holomorphic matrices}
Our next objective is to describe a pair \,$(X',\Sss')$\, such that the lift of \,$\Sss'$\, to the \,$\Sss'$-halfway normalisation of \,$X'$\, is not locally free. For this purpose, we introduce a general construction of pairs \,$(X',\Sss')$\,. This construction also plays a prominent role in the theory of integrable systems. For many integrable systems, the construction is used to associate to a solution of the related differential equation such a pair \,$(X',\Sss')$\, of so-called spectral data. It is a typical result in this theory that this construction yields a 1--1 correspondence between solutions and spectral data, see for example \cite{Hi,Sch}.  

Let \,$A: Y \to \C^{n\times n}$\, be a holomorphic map from a singular curve \,$Y$\, into the space of complex \,$(n\times n)$-matrices. In most applications, \,$Y$\, is either \,$\C$\, or \,$\C^*$\, and \,$A$\, extends in some sense to \,$\P$\,. The singular curve \,$X'$\, is defined as 
\begin{equation}
\label{eq:eigenvaluecurve}
X' = \, \bigr\{ \, (\upsilon,\mu) \in Y\times \C \ |\ \det\bigr(\mu\cdot \unity-A(\upsilon)\bigr)=0 \, \bigr\} \; . 
\end{equation}
We make the overall hypothesis that the zero set of the discriminant of the characteristic polynomial is discrete in \,$Y$\,. We denote the complement of this set by \,$Y_0$\,. In particular, for \,$(\upsilon,\mu)\in (Y_0\times\mathbb{C})\cap X'$, the eigenspace of \,$A(\upsilon)$\, with eigenvalue \,$\mu$\, is one-dimensional. We fix a linear form \,$\ell: \C^n \to \C$\, such that the set of those \,$(\upsilon,\mu) \in X'$\, is discrete for which the kernel of \,$\ell$\, contains non-trivial eigenvectors of \,$A(\upsilon)$\, with eigenvalue \,$\mu$\,. There exists a unique global meromorphic function \,$\psi = (\psi_1,\dotsc,\psi_n): X' \to \C^n$\, with
$$
A \cdot \psi = \mu \cdot \psi \quad \text{and} \quad 
   \ell(\psi) = 1,
$$ 
where we regard \,$A$\,, \,$\mu$\, and \,$\psi$\, as functions on \,$X'$\,. Locally, \,$\psi$\, can be obtained from any holomorphic eigenfunction \,$\varphi$\, by taking \,$\psi=\varphi/\ell(\varphi)$\,. The corresponding generalised divisor \,$\Sss'$\, on \,$X'$\, is gene- rated by \,$\psi_1,\dotsc,\psi_n$\,. If the eigenspaces of \,$A$\, define a line bundle on \,$X'$\,, then \,$\Sss'$\, describes the dual eigenline bundle in the sense of the correspondence between divisors and line bundles, see \cite[\S 29]{Fo}. Moreover, the generalised divisor generated by the entries $\Tilde{\psi}_1,\ldots,\Tilde{\psi}_n$ of the eigenfunction $\Tilde{\psi}$ normalised by another admissible linear form $\Tilde{\ell}$ is obtained from $\Sss'$ by multiplication with the global meromorphic function $\ell(\psi)/\Tilde{\ell}(\psi)=1/\Tilde{\ell}(\psi)$ of $Y$ and therefore isomorphic to $\Sss'$. Thus, the isomorphy class of $\Sss'$ does not depend on the choice of $\ell$. 
\begin{Proposition}\label{P:Phi}
\[ \Phi: (\Oss_{Y})^n\to \upsilon_\ast \Sss', \quad (f_1,\dotsc,f_n)\mapsto f_1\psi_1+\dotsc+f_n\psi_n\]
is an isomorphism of sheaves. Here, \,$\upsilon$\, denotes the projection \,$X' \to Y$\,. 
\end{Proposition}
\begin{proof}
First we prove surjectivity. Because \,$\psi_1,\dots,\psi_n$\, generates \,$\Sss'$\,, it suffices to show that the image of \,$\Phi$\, 
is a $\upsilon_\ast \Oss_{X'}$-module. Due to the Generalised Weierstra{\ss} division theorem \cite[Chapter~1, Theorem~1.19]{GPR}, the following map is an isomorphism of sheaves:
\begin{align}\label{eq:weierstrass isomorphism}
(\Oss_{Y})^n&\to\upsilon_\ast\Oss_{X'},&(f_1, \dotsc, f_n)&\mapsto f_1 + f_2\,\mu + \dotsc + f_n\,\mu^{n-1} \;.
\end{align} 
So it suffices to show that the multiplication with \,$\mu$\,, and therefore also with powers of \,$\mu$\,, acts on the image of \,$\Phi$\,. 
Due to $A\psi=\mu\psi$, multiplication with $\mu$ acts as matrix multiplication with $A$ on the column vector \,$\psi$\,. The entries of $A$ are holomorphic on $Y$, so the image of \,$\Phi$\, is a $\upsilon_\ast \Oss_{X'}$-module.

Now we show injectivity. We first consider \,$\Phi$\, at \,$y\in Y_0$. Then, \,$A(y)$\, has \,$n$\, pairwise distinct eigenvalues and \,$n$\, linearly independent eigenvectors. So the determinant of the eigenvectors is non-zero in a neighborhood of \,$y$\, and \,$\Phi$\, is injective on the stalk over \,$y$\,.

It follows that \,$\Phi$\, can be non-injective only on the stalks over the discrete set \,$Y\setminus Y_0$\,. On the other hand, every element in the kernel of \,$\Phi$\, is a germ of holomorphic functions, so the set of stalks on which \,$\Phi$\, is non-injective is open and therefore empty.
\end{proof}
Let \,$\Oss(A)$\, be the sheaf of holomorphic \,$(n\times n)$-matrices on \,$Y$\, which commute with \,$A$\,. Let \,$U\subset Y$\, be open and \,$B \in H^0(U,\Oss(A)_U)$\,. Since \,$B$\, commutes with \,$A|_U$\,, we can diagonalize \,$A|_U$\, and \,$B$\, simultaneously. Consequently, the eigenvalues of \,$B|_{U\cap Y_0}$\, define holomorphic functions on \,$\upsilon^{-1}[U\cap Y_0]$\,. These functions are locally bounded near the points of \,$U \cap (Y \setminus Y_0)$\,, and therefore extend to holomorphic functions on \,$\pi^{-1}[\upsilon^{-1}[U]] \subset X$\, by the Riemann Extension Theorem. This defines the natural homomorphism of rings 
\begin{equation}\label{eq:evhomo}
\Oss(A) \to \upsilon_* \bar{\Oss}_{X'} \; .
\end{equation}
\begin{Corollary}\label{commuting matrices}
The image of \eqref{eq:evhomo} is \,$\upsilon_*(\pi_{X(\Sss')})_\ast \Oss_{X(\Sss')}$\,. 
\end{Corollary}
In particular, the eigenvalues of an element of \,$\Oss(A)$\, define holomorphic functions on the \,$\Sss'$-halfway normalisation and vice versa, the holomorphic functions on the \,$\Sss'$-halfway normalisation are eigenvalues of elements of $\Oss(A)$.
\begin{proof}
Matrix multiplication with an element of \,$\Oss(A)$\, acts on $\psi$ like multiplication with the corresponding image in \eqref{eq:evhomo}. This shows that not only $\mu$ and $\upsilon_*\Oss_{X'}$, but the whole image of \eqref{eq:evhomo} acts on \,$\upsilon_*\Sss'$\,. So the image of \eqref{eq:evhomo} is contained in \,$\upsilon_*(\pi_{X(\Sss')})_\ast \Oss_{X(\Sss')}$\,. 

Conversely, for \,$\nu \in \upsilon_*(\pi_{X(\Sss')})_\ast \Oss_{X(\Sss')}$, the entries of \,$\nu\cdot \psi$\, are due to Proposition~\ref{P:Phi} linear combinations of \,$\psi_1,\dotsc,\psi_n$\, with coefficients in \,$\Oss_Y$\,. All these coefficients together define a holomorphic \,$(n\times n)$-matrix on \,$Y$\,. Because multiplication with \,$\nu$\, commutes with multiplication with \,$\mu$\,, this matrix commutes with \,$A$\, and belongs to \,$\Oss(A)$\,. It is mapped to \,$\mu$\, by \eqref{eq:evhomo}. This shows that the image of \eqref{eq:evhomo} contains \,$\upsilon_*(\pi_{X(\Sss')})_\ast \Oss_{X(\Sss')}$\,. 
\end{proof}
Next, we give an example of a generalised divisor \,$\Sss'$\, such that the corresponding generalised divisor on the \,$\Sss'$-halfway normalisation is not locally free. Hitchin defines in~\cite[Section~5]{Hi} the spectral data of a harmonic map from a 2-dimensional torus to $\mathbb{S}^3$. They are defined in terms of the holonomy representation of a family of flat connections. 
The harmonic map gives rise to such a family depending holomorphically on $\lambda\in\C^\ast$. The corresponding holonomy is a representation of the abelian fundamental group of the torus with values in the $2\times 2$-matrices depending holomorphically on \,$\lambda\in\C^*$\,. Hitchin defines the spectral curve as the unique singular curve such that the eigenspace bundle \cite[Section~7]{Hi} is locally free. For simplicity, we do not discuss the remarkable result \cite[Proposition~3.9]{Hi} that this spectral curve can be compactified by adding two smooth points at $\lambda=0$ and $\lambda=\infty$. Note that locally free generalised divisors are invertible sheaves, and a generalised divisor is locally free if it is the inverse of an invertible sheaf.
Hence, this spectral curve is the \,$\Sss'$-halfway normalisation of the holomorphic matrix given by one of the commuting holonomies, see Lemma~\ref{L:middleding-condition}. Moreover, since all holonomies commute, they define the same sheaf of commuting holomorphic matrices and therefore the same \,$\Sss'$-halfway normalisation due to~Corollary~\ref{commuting matrices}. Since the spectral curves of all these holonomies are hyperelliptic, Theorem~\ref{two sheeted} yields that the corresponding generalised divisors are locally free on the \,$\Sss'$-halfway normalisation. Therefore, the spectral data correspondence of Hitchin associates  to a harmonic map the pair $(X(\Sss'),\Sss'')$ which is constructed from one of the commuting holonomies. Here, $\Sss''$ is the corresponding locally free generalised divisor on the \,$\Sss'$-halfway normalisation $X(\Sss')$ in Lemma~\ref{L:projectdivisor}. We remark that not only the \,$\Sss'$-halfway normalisation of \,$X'$\,, but also the corresponding generalised divisor $\Sss''$ does not depend on the choice of the holonomy out of the family of commuting holonomies since the eigenfunction $\psi$ diagonalizes all these commuting holonomies simultaneously.

The following example shows that for integrable systems whose spectral curves are not hyperelliptic, such pairs of spectral curves with locally free divisors do not exist in general. However, since the statement of Theorem~\ref{two sheeted} is local, for all solutions of an integrable system with a holomorphic matrix (holonomy) whose eigenvalue curve has only simple branch points connecting only two sheets, the corresponding generalised divisor \,$\Sss''$\, is locally free on the \,$\Sss'$-halfway normalisation. This should be the generic situation and such pairs $(X(\Sss'),\Sss'')$ with locally free $\Sss''$ will exist in the generic case. In particular, a holomorphic matrix whose generalised divisor $\Sss''$ is not locally free on the \,$\Sss'$-halfway normalisation $X(\Sss')$ is at least a $3\times 3$-matrix. Furthermore, the eigenvalue curve of such a matrix has a singular point which is a branch point connecting at least three sheets. Now, it is not difficult to give an example.
\begin{Example}\cite[Example~9.3]{Sch}
Consider for \,$\lambda \in \C = Y$\, the holomorphic matrix
	\[
	A(\lambda) = \begin{pmatrix}
	\lambda &  0 & a \\
	0 &  0 & b\\
	0& 0& -\lambda
	\end{pmatrix} \;  
	\]
with fixed \,$a,b\in \C$\,. The corresponding eigenvalue curve as in Equation~\eqref{eq:eigenvaluecurve} is given by
$$ X' = \left\{ \left. \, (\lambda,\mu) \in \C^2 \,\right|\, (\mu-\lambda)\mu(\mu+\lambda)=0 \, \right\} \; . $$
The normalisation of \,$X'$\, is 
\begin{align*}
	X & = \left\{ \left. \, (\lambda,\mu)  \,\right|\, \mu=\lambda\, \right\}
	\;\dot{\cup}\;
	\left\{ \left. \, (\lambda,\mu)  \,\right|\, \mu=0 \, \right\}
	\;\dot{\cup}\;
	\left\{ \left. \, (\lambda,\mu) \,\right|\, \mu=-\lambda \, \right\}
	\\
	& \cong \C \;\dot{\cup}\; \C \;\dot{\cup}\; \C \; .
\end{align*}
The three copies of \,$\C$\, intersect in \,$X'$\, in the triple point $q$ with \,$(\lambda,\mu) = (0,0)$ which is the only singularity of \,$X'$\,.  
Therefore, \,$\bar{\Oss}_{X',q} \cong (\C\{\lambda\})^3$\,. Then, \,$\Oss_{X',q}$\, is the image of the algebra homomorphism \,$\C\{\lambda,\mu\} \to (\C\{\lambda\})^3$\, with \,$1 \mapsto (1,1,1)$\,, \,$\lambda \mapsto (\lambda,\lambda,\lambda)$\, and \,$\mu \mapsto (\lambda,0,-\lambda)$\,. We have \,$\lambda^2 \mapsto (\lambda^2,\lambda^2,\lambda^2)$\,, \,$\lambda\mu \mapsto (\lambda^2,0,-\lambda^2)$\, and \,$\mu^2 \mapsto (\lambda^2,0,\lambda^2)$\,, and thus all monomials of degree \,$2$\, in \,$(\C\{\lambda\})^3$\, belong to the image. Moreover, by multiplication with powers of \,$\lambda$, all monomials of higher degree are obtained in the image. 
Hence, the cokernel in \,$(\C\{\lambda\})^3$\, of this homomorphism is spanned by \,$(1,0,0)$\,, \,$(0,1,0)$\, and \,$(0,\lambda,0)$\, and \,$\delta_q=3$\,. 

For the construction of the eigenline bundle, we choose \,$\ell(\psi)=\psi_1+\psi_2+\psi_3$\,. We will see that this choice of \,$\ell$\, is admissible, unlike the simpler choices \,$\ell(\psi)=\psi_k$\, for \,$k\in \{1,2,3\}$\,.
Then, \,$\psi$\, is the solution of the following linear system of equations:
\begin{align}
	\label{eq:I}
		\lambda\psi_1 + a \psi_3&=\mu\psi_1 \\ 
	\label{eq:II}
		b\psi_3&=\mu\psi_2\\
	\label{eq:III}
		-\lambda \psi_3& = \mu\psi_3\\
	\label{eq:IV}
		\psi_1+\psi_2+\psi_3&=1 \; . 
\end{align}
We now calculate the solution \,$\psi$\, individually on the three components of the normalisation \,$X$\,: 

\noindent
For $\mu=\lambda$, one gets from \eqref{eq:I} that $a\psi_3=0$ and $\psi_3=0$ for \,$a\neq 0$\,. So \eqref{eq:II} implies $\mu\psi_2=0$ and thus \,$\psi_2=0$\,. By \eqref{eq:IV}, we obtain $\psi_1=1$, so $\psi=(1,0,0)^t$\,. This excludes the linear forms \,$\ell(\psi)=\psi_k$\, for \,$k\in \{2,3\}$\, from being admissible.

\noindent
For $\mu=0$, \eqref{eq:III} gives $-\lambda \psi_3=0$ and hence  $\psi_3=0$. Inserting this into \eqref{eq:I} gives $\psi_1=0$ and \eqref{eq:IV} yields $\psi_2=1$. So \,$\psi=(0,1,0)^t$\, which excludes the linear forms \,$\ell(\psi)=\psi_k$\, for \,$k\in \{1,3\}$\, from being admissible.

\noindent
For $\mu=-\lambda$, \eqref{eq:I} reads as $2\lambda \psi_1=-a\psi_3$ and \eqref{eq:II} reads as $b\psi_3=-\lambda\psi_2$, hence
		\[
		\psi_1=-\frac{a}{2\lambda}\psi_3 \quad \mbox{ and } \quad \psi_2=-\frac{b}{\lambda}\psi_3.
		\]
		Then \eqref{eq:IV} yields $\psi_3\left(-\frac{a}{2\lambda}-\frac{b}{\lambda}+1\right)=1$ and therefore
\begin{align*}
\psi&=\left(-\frac{a}{2\lambda},-\frac{b}{\lambda},1\right)^t\left(-\frac{a}{2\lambda}-\frac{b}{\lambda}+1\right)^{-1}=\frac{(-a,-2b,2\lambda)^t}{2\lambda-a-2b}.
\end{align*}
Hence, $\psi$ has a pole at $\lambda=\frac{a+2b}{2}$ which is for \,$a \neq -2b$\, a point in the regular set of \,$X'$\,. 
		
So we have determined the generators \,$\psi_1,\psi_2,\psi_3$\, of the generalised divisor \,$\Sss'$\, associated to the holomorphic matrix \,$A$\,. Our next objective is to show that the corresponding generalised divisor on the \,$\Sss'$-halfway normalisation is not locally free. For this purpose, we extend the isomorphism~\eqref{eq:weierstrass isomorphism} to
\begin{align*}
(\Mss_{Y})^n&\to\upsilon_\ast\Mss_{X'},&(f_1, \dotsc, f_n)&\mapsto f_1 + f_2\,\mu + \dotsc + f_n\,\mu^{n-1} \;.
\end{align*}
Since the product of any meromorphic function on $X'$ with the pullback of an appropriate function on $Y$ becomes locally holomorphic, the arguments in the proof of Proposition~\ref{P:Phi} extend to the meromorphic functions and show surjectivity and injectivity of this sheaf homomorphism. Thus, due to Corollary~\ref{commuting matrices}, in our example \,$\lambda_*(\pi_{X(\Sss')})_\ast \Oss_{X(\Sss')}$\, is equal to
$$\{f_1+f_2\mu+f_3\mu^2\mid(f_1,f_2,f_3)\in\Mss_{\mathbb{C}}^3\mbox{ with }f_1\unity+f_2A+f_3A^2\in\Oss_\mathbb{C}^{n\times n}\}.$$
To determine such $(f_1,f_2,f_3)\in\Mss_{\mathbb{C},0}^3$, first insert
\[f_1\unity+f_2A+f_3A^2=\begin{pmatrix}
f_1+f_2\lambda+f_3\lambda^2 & 0 & f_2a\\
0 & f_1 & f_2b-f_3b\lambda\\
0 & 0 & f_1-f_2\lambda+f_3\lambda^2
\end{pmatrix}.
\]
This implies for $a\neq 0$ that $af_2$ and $f_2$ are holomorphic. If in addition $b\neq 0$, then $b\lambda f_3$ and $\lambda f_3$ are holomorphic. Considering the middle entry in the matrices of the above equation implies that also $f_1$ is holomorphic. Hence, the stalk of $(\pi_{X(\Sss')})_\ast \Oss_{X(\Sss')}$ at $q$ is the image of
\[
\C\left\{\lambda, \mu,\mu^2\lambda^{-1}\right\}\to\Bar{\Oss}_{X',q}.
\]
For any generalised divisor $\Sss'$ whose stalk at $q$ is contained in $\Bar{\Oss}_{X',q}$, the evaluation at the three points in $\pi^{-1}[\{q\}]$ defines a linear map $\Sss'_q\to\mathbb{C}^3$. Since all generators of $\mathbb{C}\{\lambda,\mu,\mu^2\lambda^{-1}\}$ vanish at $q$, the image of this map is at most one-dimensional if the divisor is locally free with respect to $\mathbb{C}\{\lambda,\mu,\mu^2\lambda^{-1}\}$. For $a\ne0$, $b\ne0$ and $a+2b\ne 0$, the values
\begin{equation*}
\psi=\begin{cases}(1,0,0)^t&\mbox{ at }(\lambda,\mu)=(0,0)\mbox{ in }\mu=\lambda\\(0,1,0)^t&\mbox{ at }(\lambda,\mu)=(0,0)\mbox{ in }\mu=0\\(\tfrac{-a}{-a-2b},\tfrac{-2b}{-a-2b},0)^t&\mbox{ at }(\lambda,\mu)=(0,0)\mbox{ in }\mu=-\lambda\end{cases}
\end{equation*}
span a 2-dimensional space. We are not considering $a=0$ or $b=0$ here. In these cases, one can determine the corresponding sheaf $\Oss(A)$ by the same methods, but the ring is larger and the calculations are more involved. For $a\ne0$, $b\ne0$ and $a+2b\ne0$, the corresponding generalised divisor $\Sss''$ on the \,$\Sss'$-halfway normalisation $X(\Sss')$ is not locally free.
\end{Example}
This example might be considered somewhat pathological since the normalisation has three connected components. But all arguments only depend on the local behaviour of $A$ in a neighborhood of the triple point. More precisely, all arguments carry over to triple points where the holomorphic matrix in the neighborhood of this point differs from $A$ by a matrix with a second order root.
\section{Riemann Roch}\label{se:riemann roch}
From now on, we suppose that $X'$ is compact, or equivalently that $X$ is compact, unless specified otherwise. Let $g=\dim H^1(X,\mathcal{O}_X)<\infty$ be the genus of $X$ which is called the geometric genus of $X'$ and define the $\delta$-invariant of $X'$ as
\[
\delta:=\sum_{q\in S'}\delta_q\;.
\]
Since $S'$ is discrete, one has $\delta<\infty$.

\begin{Lemma} 
\label{L:dimH}
\strut

\vspace{-1ex} 

\begin{itemize}
\item[(a)]$\dim H^0(X,\mathcal{O}_X)=\dim H^0(X',\bar{\mathcal{O}}_{X'})$, \\
 $\dim H^1(X,\mathcal{O}_X)=\dim H^1(X',\bar{\mathcal{O}}_{X'})$.
\item[(b)]$\dim H^1(X',\mathcal{O}_{X'})=g+\delta<\infty$.
This number is called arithmetic genus of $X'$ and is denoted as $g'$.
\end{itemize}
\end{Lemma}

\begin{proof}
\emph{(a)}
Since $\bar{\mathcal{O}}_{X'}=\pi_*\mathcal{O}_X$, we have $H^0(U',\bar{\mathcal{O}}_{X'})=H^0(\pi^1[U'],\mathcal{O}_X)$ for every open subset $U'\subset X'$. By definition of the \v{C}ech complex, this implies $H^q(\mathcal{U}',\bar{\mathcal{O}}_{X'})=H^q(\pi^{-1}[\mathcal{U}'],\mathcal{O}_X)$ for every open covering $\mathcal{U}'$ of $X'$ and $q\in\{0,1\}$. Here, $\pi^{-1}[\mathcal{U}']$ denotes the covering $X=\bigcup_{U'\in\mathcal{U}'}\pi^{-1}[U']$. Let all $U'\in\mathcal{U}'$ be non-compact. Then, all $\pi^{-1}[U']\in\pi^{-1}[\mathcal{U}']$ are also non-compact and thus, $\pi^{-1}[\mathcal{U}']$ is a Leray covering, see \cite[Theorem~12.8 and Theorem~26.1]{Fo}. Hence, we have $H^q(X,\mathcal{O}_X)=H^q(\pi^{-1}[\mathcal{U}'],\mathcal{O}_X)$. Because $\mathcal{U}'\mapsto\pi^{-1}[\mathcal{U}']$ embeds the ordered set of coverings of $X'$ into the ordered set of coverings of $X$, the inductive limit of $H^q(\mathcal{U}',\bar{\mathcal{O}}_{X'})$ is a sublimit of the inductive limit of $H^q(\pi^{-1}[\mathcal{U}'],\mathcal{O}_X)$. So the latter inductive limit is equal to $H^q(\pi^{-1}[\mathcal{U}'],\mathcal{O}_X)$ for any Leray covering and therefore also the former inductive limit.

\emph{(b)}
The sequence of sheaves 
\[
0\longrightarrow\mathcal{O}_{X'}\longrightarrow\bar{\mathcal{O}}_{X'}\longrightarrow \bar{\mathcal{O}}_{X'}/\mathcal{O}_{X'}\longrightarrow 0
\]
is exact. So we have the alternating sum of dimensions
\begin{multline}\label{eq:Hq-alternating-sum}
\dim H^1(X',\mathcal{O}_{X'}) = \dim H^0(X',\mathcal{O}_{X'}) \, - \\ - \dim H^0(X',\bar{\mathcal{O}}_{X'}) 
 + \dim H^0(X',\bar{\mathcal{O}}_{X'}/\mathcal{O}_{X'}) + \\ + \dim H^1(X',\bar{\mathcal{O}}_{X'})-\dim H^1(X',\bar{\mathcal{O}}_{X'}/\mathcal{O}_{X'}) \; .  
\end{multline}
Due to (a), $\dim H^0(X',\bar{\mathcal{O}}_{X'})=1$ and $\dim H^1(X',\bar{\mathcal{O}}_{X'})=g$, and Proposition~\ref{prop1}~(a) yields that $\dim H^0(X',\bar{\mathcal{O}}_{X'}/\mathcal{O}_{X'})=\delta$. 
The support of $\bar{\mathcal{O}}_{X'}/\mathcal{O}_{X'}$ is discrete, see Proposition~\ref{prop3}, so $H^1(X',\bar{\mathcal{O}}_{X'}/\mathcal{O}_{X'})=0$.

$\mathcal{O}_{X'}$ is a subsheaf of $\bar{\mathcal{O}}_{X'}$ containing the constant functions, so we have $\dim H^0(X',\mathcal{O}_{X'})=\dim H^0(X',\bar{\mathcal{O}}_{X'})=\dim H^0(X,\mathcal{O}_X)=1$. Inserting all this into \eqref{eq:Hq-alternating-sum} gives the assertion.
\end{proof}
The Riemann Roch theorem now follows easily from the definition of the degree of a generalised divisor. 
\begin{Theorem}
\label{T:riemannroch}
For every generalised divisor $\Sss$ on $X'$ one has
\[
\dim H^0(X',\Sss)-\dim H^1(X',\Sss)=\deg \Sss +1 - g' \;.
\]
\end{Theorem}
\begin{proof}
Let $\Sss'$ be a generalised divisor with $\Sss\subset \Sss'$ and $\Oss_{X'}\subset \Sss'$. Then the sequences
\[
0\rightarrow \Sss \hookrightarrow \Sss' \rightarrow \Sss'/\Sss \rightarrow 0
\]
and
\[
0\rightarrow \Oss_{X'} \hookrightarrow \Sss' \rightarrow \Sss'/\Oss_{X'} \rightarrow 0
\] 
are exact. Due to Proposition~\ref{prop3}, $\supp(\Sss'/\Sss) \subset \supp(\Sss') \cup \supp(\Sss)$ and $\supp(\Sss'/\Oss_{X'})=\supp(\Sss')$ are discrete. So one has $H^1(X',\Sss'/\Sss)=0$ and $H^1(X',\Sss'/\Oss_{X'})=0$, and the long exact sequences are given by
\begin{multline*}
0\rightarrow H^0(X',\Sss)\to H^0(X',\Sss')\to H^0(X',\Sss'/\Sss) \to \\
\to H^1(X',\Sss)\to H^1(X',\Sss')\to 0
\end{multline*}
and 
\begin{multline*}
0 \to H^0(X',\Oss_{X'})\to H^0(X',\Sss')\to H^0(X',\Sss'/\Oss_{X'}) \to \\
\to H^1(X',\Oss_{X'}) \to H^1(X',\Sss') \to 0 \; .
\end{multline*}
So the alternating sums of dimensions of spaces in exact sequences yield
\begin{multline}\label{eq:rr1}
\dim H^0(X',\Sss)-\dim H^0(X',\Sss') + \overbrace{\dim H^0(X',\Sss'/\Sss)}^{= \deg (\Sss') -\deg (\Sss)} \\
-\dim H^1(X',\Sss)+\dim H^1(X',\Sss')=0
\end{multline}
and 
\begin{multline}\label{eq:rr2}
\overbrace{\dim H^0(X',\Oss_{X'})}^{= 1}-\dim H^0(X',\Sss') + \overbrace{\dim H^0(X',\Sss'/\Oss_{X'})}^{= \deg (\Sss')} \\
	- \underbrace{\dim H^1(X',\Oss_{X'})}_{=g'}+\dim H^1(X',\Sss')=0 \;.
\end{multline}
Subtracting~\eqref{eq:rr2} from~\eqref{eq:rr1} yields the assertion.
\end{proof}
A simple consequence of the Riemann-Roch theorem is the existence of global meromorphic functions on \,$X'$\,. 
\begin{Corollary}
\label{C:meromorphic-functions}
For any generalised divisor \,$\Sss'$\, on \,$X'$\, with \,$\deg \Sss' \geq g$, we have
\,$\dim H^0(X',\Sss') \geq 1$\,.
In particular, for every generalised divisor \,$\Sss'$, there exists a generalised divisor  \,$\widetilde{\Sss}' \supset \Sss'$\, with \,$\mathrm{supp}(\widetilde{\Sss}'/\Sss') \subset X'\setminus S'$\, and \,$\dim H^0(X',\widetilde{\Sss}') \geq 1$\,.\qed
\end{Corollary}

The elements of \,$H^0(X',\widetilde{\Sss}')$\, are meromorphic sections of \,$\Sss'$\, because they are holomorphic with the exception of finitely many smooth poles.
\section{Regular differential forms and Serre duality}\label{se:serre duality}
Next, we define regular differential forms on $X'$. This definition also makes sense if $X'$ is non-compact.
\begin{Definition}
A meromorphic differential form $\omega$ on $X'$ is \emph{regular at $q\in X'$} if
\begin{equation}\label{eq:Res}
\sum_{p\in\pi^{-1}[\{q\}]}\Res_p\big(\pi^*(f\cdot \omega)\big)=0 
\ \text{ for all }f\in \mathcal{O}_{X',q} \; , 
\end{equation}
where the residue \,$\Res_p$\, is defined as in \cite[\S 9.9]{Fo}.
We say that $\omega$ is \emph{regular} if $\omega$ is regular at every $q\in X'$. $\Omega_{X'}$ is the \emph{sheaf of regular differential forms} on $X'$.
\end{Definition}

We denote the \emph{sheaf of holomorphic differential forms} on $X$ by $\Omega_X$. The stalk at $q\in X'$ of the direct image $\pi_*\Omega_X$ is defined by 
\[
(\pi_*\Omega_X)_q=\bigoplus_{p\in\pi^{-1}[\{q\}]}\Omega_{X,p}\; .
\]
Denote the sheaf of meromorphic $1$-forms on $X'$ by $d\mathcal{M}_{X'}$.
Similarly as for meromorphic functions, the sheaf $d\mathcal{M}_{X'}$
can be identified with the direct image sheaf $\pi_* d\mathcal{M}_{X}$\,. For germs $\omega\in(\pi_*\Omega_X)_q$ \eqref{eq:Res} holds for all $f \in \bar{\mathcal{O}}_{X',q}$, so $\pi_*\Omega_X$ is a subsheaf of $\Omega_{X'}$. Therefore, a \emph{pairing} between $\bar{\mathcal{O}}_{X',q}/\mathcal{O}_{X',q}$ and $\Omega_{X',q}/(\pi_*\Omega_X)_q$ is defined by
\begin{equation}\label{eq:Res2}
(f,\omega)\mapsto\sum_{p\in\pi^{-1}[\{q\}]}\Res_p\big(\pi^*(f\cdot\omega)\big).
\end{equation}
\begin{Lemma}
\label{L:nondegenerate}
The pairing in \eqref{eq:Res2} is non-degenerate.
\end{Lemma}
\begin{proof}
Let $f\in(\bar{\mathcal{O}}_{X',q}/\mathcal{O}_{X',q})\setminus\{0\}$. Because this quotient is finite-dimensional, there exists a linear form on $\bar{\mathcal{O}}_{X',q}$ which vanishes on  $\mathcal{O}_{X',q}$ but is non-zero at $f$. The pairing between $\mathcal{M}_{X',q}$ and the stalk of meromorphic differential forms on $X'$ at $q$ is non-degenerate. Therefore, this linear form can be represented by the pairing of $f$ with some $\omega$ from this stalk. Consequently, the pairing of $\Oss_{X',q}$ with $\omega$ vanishes. So $\omega\in \Omega_{X',q}$, and the pairing \eqref{eq:Res2} of $f$ with $\omega$ is non-zero. 

Conversely, let $\omega\in\big(\Omega_{X',q}/(\pi_*\Omega_X)_{q}\big)\setminus\{0\}$ be given. Then, $\omega$ has a pole at some $p\in\pi^{-1}[\{q\}]$. Due to Equation~\eqref{eq:decompose-Obar}, the pairing \eqref{eq:Res2} of $\omega$ with some $f\in\bar{\mathcal{O}}_{X',q}$ is non-zero. 
\end{proof}

It follows from this lemma that also $\dim(\Omega_{X',q}/(\pi_*\Omega_X)_{q}\big)=\delta_q$ holds. Because $\Omega_X$ is coherent on $X,\pi_*\Omega_X$ is coherent on $X'$, and therefore also $\Omega_{X'}$. 

Every global meromorphic function $f$ on $X'$ which does not vanish identically on a connected component of $X'$ has an inverse meromorphic function. For such an $f$, the map $g\mapsto g\cdot df$ is an isomorphism from the sheaf of meromorphic functions onto the sheaf of meromorphic $1$-forms. In this sense, we can interpret finitely generated $\mathcal{O}_{X'}$-submodules of the sheaf of meromorphic $1$-forms as generalised divisors on $X'$. 
More precisely, the map \,$g \cdot df \mapsto g$\, identifies finitely generated \,$\Oss_{X'}$-submodules of \,$d\mathcal{M}_{X'}$\, with generalised divisors on \,$X'$\, and the isomorphism classes of the generalised divisors do not depend on the choice of \,$df$\,. The degree of such a submodule is defined as the degree of the corresponding generalised divisor.

\begin{Theorem}
\label{T:Omega-gendiv}
For every generalised divisor $\mathcal{S}'$ on $X'$,
\[
\Omega_{X'}(\mathcal{S}'):=\{\omega\in d\mathcal{M}_{X'}\,\mid\, f\cdot\omega\in\Omega_{X'}\text{ for all }f\in\mathcal{S}'\}
\]
is a generalised divisor on $X', i.e.\ \Omega_{X'}(\mathcal{S}')$ is a finitely generated $\mathcal{O}_{X'}$-submodule of $d\mathcal{M}_{X'}$\,. 
\end{Theorem}

\begin{proof}
Let $q\in X'$, $\omega\in \Omega_{X',q}(\mathcal{S}')$, $g\in\mathcal{O}_{X',q}$ and $f\in\mathcal{S}_q'$. Then $f\cdot g\in\mathcal{S}_q'$ and hence, $f\cdot g\cdot\omega\in\Omega_{X',q}$. Since this holds for all $f\in \mathcal{S}_q'$, we have $g\cdot\omega\in\Omega_{X',q}(\mathcal{S}')$. This shows that $\Omega_{X'}(\mathcal{S}')$ is an $\mathcal{O}_{X'}$-submodule of $d\mathcal{M}_{X'}$. In order to prove that it is finitely generated, it suffices to show that $\Omega_{X',q}(\mathcal{S}')$ is a finitely generated $\bar{\mathcal{O}}_{X',q}$-module because $\delta_q<\infty$. Let $f_1,\ldots,f_n$ generate $\mathcal{S}_q'$, and for \,$p\in \pi^{-1}[\{q\}]$\,, let $\mathcal{S}_p$ be the $\mathcal{O}_{X,p}$-submodule of $\mathcal{M}_{X,p}$ generated by $\pi^* f_1,\ldots,\pi^* f_n$. Since on $X$, all generalised divisors are locally free, in particular invertible, the following submodule of $\Omega_{X',q}(\mathcal{S}')$ is finitely generated as an $\bar{\mathcal{O}}_{X',q}$-module and therefore also as an $\mathcal{O}_{X',q}$-module:
\[
(\pi_* \Omega_X(\mathcal{S}))_q =\{\omega\in d\mathcal{M}_{X',q}\,\mid\,f\cdot\omega\in(\pi_* \Omega_X)_q\text{ for all }f\in (\pi_*\mathcal{S})_q\}\;.
\]
It suffices to show now that $\Omega_{X',q}(\mathcal{S}')/(\pi_* \Omega_X(\mathcal{S}))_q$ has finite dimension. By definition of $(\pi_*\mathcal{S})_q$, we have $\dim((\pi_* \mathcal{S})_q/\mathcal{S}_q')<\infty$. By Lemma \ref{L:nondegenerate}, $\dim(\Omega_{X',q}/(\pi_*\Omega_X)_q)=\delta_q<\infty$. This implies the claim.
\end{proof}

\begin{Theorem}[Serre duality]
\label{T:serre}
For every generalised divisor $\mathcal{S}'$, the vector space $H^0(X',\Omega_{X'}(\mathcal{S}'))$ is canonically isomorphic to the dual of $H^1(X',\mathcal{S}')$ and $H^1(X',\Omega_{X'}(\mathcal{S}'))$ to the dual of $H^0(X',\mathcal{S}')$.
\end{Theorem}

\begin{proof}
We transfer the proof of Serre duality for Riemann surfaces in \cite[\S 17]{Fo} and use that it holds on $X$. We will be able to localize the additional ingredients at the singularities. As in the proof of Theorem \ref{T:Omega-gendiv}, let $\mathcal{S}$ be the divisor on $X$ which is locally generated by the pullbacks $\pi^* f_i$ of local generators $f_i$ of $\mathcal{S}'$. Moreover, $\Omega_X(\mathcal{S})$ denotes the sheaf of meromorphic forms on $X$ whose products with elements of $\mathcal{S}$ are holomorphic. Due to Serre duality on $X$, the vector space $H^0(X,\Omega_X(\mathcal{S}))$ is canonically isomorphic to the dual of $H^1(X,\mathcal{S})$. The arguments of the proof of Lemma \ref{L:dimH}(a) show in the present situation
\begin{equation}
\label{eq:Serre-H0H1}
H^0(X,\Omega_X(\mathcal{S})\!)\!\cong\!H^0(X',\pi_*\Omega_X(\mathcal{S})\!)\text{ and }
H^1(X,\mathcal{S})\!\cong\!H^1(X',\pi_*\mathcal{S}).
\end{equation}
To proceed, we consider the following exact sequences of sheaves on $X'$:
\[
0\longrightarrow\mathcal{S}'\longrightarrow\pi_*\mathcal{S}\longrightarrow\pi_*\mathcal{S}/\mathcal{S}'\longrightarrow 0,
\]
\[
0\longrightarrow\pi_*\Omega_X(\mathcal{S})\longrightarrow\Omega_{X'}(\mathcal{S}')\longrightarrow\Omega_{X'}(\mathcal{S}')/\pi_*\Omega_X(\mathcal{S})\longrightarrow 0.
\]
The supports of the sheaves $\pi_*\mathcal{S}/\mathcal{S}'$ and $\Omega_{X'}(\mathcal{S}')/\pi_*\Omega_X(\mathcal{S})$ are contained in $S'$ and therefore discrete. This implies
$$H^1(X',\pi_*\mathcal{S}/\mathcal{S}')=0=H^1(X',\Omega_{X'}(\mathcal{S}')/\pi_*\Omega_X(\mathcal{S})).$$
The corresponding long exact sequences read as
\begin{multline*}
0\rightarrow H^0(X',\mathcal{S}')\hookrightarrow H^0(X',\pi_*\mathcal{S})\rightarrow H^0(X',\pi_*\mathcal{S}/\mathcal{S}')\rightarrow\\\rightarrow H^1(X',\mathcal{S}')\twoheadrightarrow H^1(X',\pi_*\mathcal{S})\rightarrow 0,
\end{multline*}
\begin{multline*}
\hspace{-3mm}0\!\rightarrow\!H^0(X'\!,\pi_*\Omega_X(\mathcal{S})\!)\!\hookrightarrow\!H^0(X'\!,\Omega_{X'}(\mathcal{S}')\!)\!\rightarrow\!H^0(X'\!,\Omega_{X'}(\mathcal{S}')/\pi_*\Omega_X(\mathcal{S})\!)\!\rightarrow\!\\\rightarrow H^1(X',\pi_*\Omega_X(\mathcal{S}))\twoheadrightarrow H^1(X',\Omega_{X'}(\mathcal{S}'))\rightarrow 0. 
\end{multline*}
By dualizing the second exact sequence, we obtain the horizontal sequences of the following diagram:
\begin{multline}
\begin{array}{cccccccc}
\hspace{-5mm}0&\hspace{-2mm}\rightarrow\hspace{-2mm}&H^0(X',\mathcal{S}')&\hspace{-2mm}\hookrightarrow\hspace{-2mm}&H^0(X',\pi_*\mathcal{S})&\hspace{-2mm}\rightarrow\hspace{-2mm}&H^0(X',\pi_*\mathcal{S}/\mathcal{S}')&\hspace{-2mm}\rightarrow\\
\hspace{-5mm}\downarrow&&\downarrow f_0&&\downarrow f_1&&\downarrow f_2\\
\hspace{-5mm}0&\hspace{-2mm}\rightarrow\hspace{-2mm}&\hspace{-2mm}H^1(X'\!,\Omega_{X'}(\mathcal{S}')\!)^\ast\hspace{-3mm}&\hspace{-2mm}\rightarrow\hspace{-2mm}&\hspace{-2mm}H^1(X'\!,\pi_*\Omega_X(\mathcal{S})\!)^\ast\hspace{-3mm}&\hspace{-2mm}\rightarrow\hspace{-2mm}&\hspace{-2mm}H^0(X'\!,\Omega_{X'}(\mathcal{S}')/\pi_*\Omega_X(\mathcal{S})\!)^\ast\hspace{-3mm}&\hspace{-2mm}\rightarrow
\vspace{3mm}\end{array}\\\begin{array}{cccccc}
\rightarrow\hspace{-2mm}&\hspace{-2mm}H^1(X',\mathcal{S}')\hspace{-2mm}&\hspace{-2mm}\twoheadrightarrow\hspace{-2mm}&\hspace{-2mm}H^1(X',\pi_*\mathcal{S})\hspace{-2mm}&\hspace{-2mm}\rightarrow\hspace{-2mm}&\hspace{-1mm}0\hspace{-3mm}\\
&\downarrow f_3&&\downarrow f_4&&\hspace{-2mm}\downarrow\hspace{-3mm}\\
\rightarrow\hspace{-2mm}&\hspace{-2mm}H^0(X'\!,\Omega_{X'}(\mathcal{S}')\!)^\ast\hspace{-3mm}&\hspace{-2mm}\twoheadrightarrow\hspace{-2mm}&\hspace{-2mm}H^0(X'\!,\pi_*\Omega_X(\mathcal{S})\!)^\ast\hspace{-3mm}&\hspace{-2mm}\rightarrow\hspace{-2mm}&\hspace{-1mm}0\hspace{-3mm}
\end{array}\label{eq:diagram}\end{multline}
We now supplement the vertical morphisms $f_0,\ldots,f_4$ such that the whole diagram is commutative and $f_1$, $f_2$ and $f_4$ are isomorphisms. Then $f_0$ and $f_3$ also are isomorphisms due to the $5$-Lemma, see~\cite[Lemma IV.5.10]{Bredon}. Due to Serre duality on $X$ and \eqref{eq:Serre-H0H1}, there exist canonical isomorphisms $f_1$ and $f_4$. It remains to define $f_0$, $f_2$ and $f_3$ such that the diagram is commutative and to prove that $f_2$ is an isomorphism. To define $f_0$, $f_2$ and $f_3$, we choose an open covering $\mathcal{U}$ of $X'$ whose pullback to $X$ is a Leray covering.

The map $f_2$ is defined in terms of the pairings
\begin{align}\label{eq:serre-pairing}
(\pi_*(\mathcal{S}))_q/\mathcal{S}'_q\times\Omega_{X',q}(\mathcal{S}')/(\pi_*\Omega_X(\mathcal{S}))_q &\rightarrow\mathbb{C}\nonumber\\
(f,\omega)&\mapsto\hspace{-4mm}\sum_{p\in\pi^{-1}[\{q\}]}\Res_p(\pi^*(f\cdot\omega)). 
\end{align}
By definition of $\Omega_{X'}$, we have
\[
\Omega_{X',q}(\mathcal{S}')=\biggr\{\omega\in d\mathcal{M}_{X',q}\ \biggr|\ \forall \;f\in\mathcal{S}'_q : \sum_{p\in\pi^{-1}[\{q\}]}\Res_p(\pi^*(f\cdot\omega))=0\biggr\}\;.
\]
One also has
\[
(\pi_*\Omega_X(\mathcal{S}))_q=\biggr\{\omega\in d\mathcal{M}_{X',q}\;\biggr|\;\forall f\in(\pi_*\mathcal{S})_q :\hspace{-4mm} \sum_{p\in\pi^{-1}[\{q\}]}\hspace{-3mm}\Res_p(\pi^*(f\cdot\omega))=0\;\biggr\} \; . 
\]
Hence, the above pairing is well-defined on the quotient spaces.

The vector spaces $H^0(X',\pi_*\mathcal{S}/\mathcal{S}')$ and $H^0(X',\Omega_{X'}(\mathcal{S}')/\pi_*\Omega_X(\mathcal{S}))$ are the direct sums over $q\in S'$ of the corresponding stalks. Therefore, these pairings extend to a pairing between these vector spaces where elements with disjoint support are orthogonal. Furthermore, the map $f_2$ induced by these pairings is an isomorphism if the pairings in \eqref{eq:serre-pairing} are non-degenerate for all $q\in S'$. The proof of the non-degeneracy of the pairing is similar to the proof of Lemma~\ref{L:nondegenerate}. The dimension of $(\pi_*\mathcal{S})_q/\mathcal{S}'_q$ is at most $\delta_q$ times the number of generators of $\mathcal{S}'_q$ and hence finite. Similarly, $\dim(\Omega_{X',q}(\mathcal{S}')/(\pi_*\Omega_X(\mathcal{S}))_q)<\infty$. The non-degeneracy in both entries of the pairing now follows as in the first part of the proof of Lemma~\ref{L:nondegenerate}.

We define $f_3$ in terms of a pairing between $H^0(X',\Omega_{X'}(\mathcal{S}'))$ and $H^1(X',\mathcal{S}')$. Due to \cite[Corollary 17.17]{Fo}, $H^1(X,\mathcal{M})=0$. We choose a covering $\mathcal{U}'$ of $X'$ by non-compact open sets. As before, $H^1(\mathcal{U}',\mathcal{M}_{X'})=0$ and therefore $H^1(X',\mathcal{M}_{X'})=0$. Let $\nu\in H^1(X',\mathcal{S}')$. There exists $\mu\in C^0(\mathcal{U}',\mathcal{M}_{X'})$ with $\delta\mu=\nu$. We define the pairing between $\nu$ and $\omega\in H^0(X',\Omega_{X'}(\mathcal{S}'))$ as $(\nu,\omega)\mapsto\sum_{p\in X}\Res_p(\pi^*(\mu\cdot\omega))$. We have to show that the right hand side only depends on $\nu\in H^1(X',\mathcal{S}')$, i.e.~that it does not change if we add to $\mu$ an element of $Z^0(\mathcal{U}',\mathcal{M}_{X'})=H^0(X',\mathcal{M})$. This follows from the Residue Theorem \cite[Theorem 10.21]{Fo}. So $f_3$ is defined. We define $f_0$ analogously to $f_3$.

It remains to prove the commutativity of the diagram \eqref{eq:diagram}.

Obviously, it commutes at the left of $f_0$, and at the right of $f_4$. To show commutativity between $f_3$ and $f_4$, let $\nu\in H^1(X',\mathcal{S}')$ and $\mu\in C^0(\mathcal{U}',\mathcal{M}_{X'})$ as in the definition of $f_3$. Then the image of $\nu$ in $H^1(X',\mathcal{S}')\to H^1(X',\pi_*\mathcal{S})$ is also the coboundary of $\mu$ in $B^1(\mathcal{U}',\mathcal{M}_{X'})$. Now, the pairing between this image of $\nu$ and $\omega\in H^0(X',\pi_*\Omega_X(\mathcal{S}))$ equals the pairing between $\nu$ and the image of $\omega$ in $H^0(X',\pi_*\Omega_X(\mathcal{S}))\hookrightarrow H^0(X',\Omega_{X'}(\mathcal{S}'))$. This implies commutativity between  $f_3$ and $f_4$. 
Analogously one shows the commutativity between $f_0$ and $f_1$.

To show commutativity between $f_1$ and $f_2$, choose $g\in H^0(X',\pi_*\mathcal{S})$ and $\alpha\in\Omega_{X',q}(\mathcal{S}')/\big(\pi_*\Omega_X(\mathcal{S})\big)_q\subset H^0\big(X',\Omega_{X'}(\mathcal{S}')/\pi_*\Omega_X(\mathcal{S})\big)$ and consider the mappings
\begin{equation*}
\varphi:H^0(X',\pi_*\mathcal{S}) \to H^0(X',\pi_*\mathcal{S}/\mathcal{S}') 
\end{equation*}
and
 \begin{equation*}
  \psi:H^0(X',\Omega_{X'}(\mathcal{S}')/\pi_*\Omega_X(\mathcal{S}))  \to H^1(X',\pi_*\Omega_X(\mathcal{S})) \;, 
\end{equation*}
where $\psi$ is the connecting homomorphism as defined in \cite[15.11]{Fo}. We have to show that the pairing between $g$ and $\psi(\alpha)$ given by the classical Serre duality via \eqref{eq:Serre-H0H1} equals the pairing between $\varphi(g)$ and $\alpha$. The latter pairing is equal to $\sum_{p\in\pi^{-1}[\{q\}]}\Res_p(\pi^*(\alpha\cdot g)).$

We track the construction of $\psi(\alpha)$ from \cite[15.11]{Fo}. There exists a neighborhood $U_q$ of $q$ in $X'$ with $U_q\cap S'=\{q\}$ together with $\beta\in H^0(U_q,\Omega_{X'}(\mathcal{S}'))$ which is mapped to $\alpha$ under 
\begin{multline*}
H^0(U_q,\Omega_{X'}(\mathcal{S}'))\to\Omega_{X',q}(\mathcal{S}')\to\\ \to\Omega_{X',q}(\mathcal{S}')/(\pi_*\Omega_X(\mathcal{S}))_q\subset H^0(X',\Omega_{X'}(\mathcal{S}')/\pi_*\Omega_X(\mathcal{S})).
\end{multline*}
We supplement $U_q$ by the open set $X'\setminus\{q\}$ to obtain the open covering $\mathcal{U}'$ of $X'$. Together with $0$ on $X'\setminus\{q\}$, $\beta$ defines an element $\gamma\in C^0(\mathcal{U}',\Omega_{X'}(\mathcal{S'}))$. On $U_q\cap(X'\setminus\{q\})=U_q\setminus\{q\}$, one has $\pi_*\Omega_X(\mathcal{S})=\Omega_{X'}(\mathcal{S}')$ and hence $\delta\gamma\in C^1(\mathcal{U}',\Omega_{X'}(\mathcal{S}'))=C^1(\mathcal{U}',\pi_*\Omega_X(\mathcal{S}))$. In fact, we have $\delta\gamma\in B^1(\mathcal{U}',\Omega_{X'}(\mathcal{S}'))$ and hence, the coboundary of $\delta\gamma\in C^1(\mathcal{U}',\pi_*\Omega_X(\mathcal{S}))$ also vanishes. Therefore, $\delta\gamma$ induces an element of $H^1(\mathcal{U}',\pi_*\Omega_X(\mathcal{S}))=H^1(X',\pi_*\Omega_X(\mathcal{S}))$, where the last equality follows because $\mathcal{U}'$ is Leray.

By definition of the connecting homomorphism, this element equals $\psi(\alpha)$. According to Serre duality on $X$ (see \cite[17.5]{Fo}), the pairing of $\psi(\alpha)$ with $g$ is equal to
\[
\sum_{p\in X} \mathrm{Res}_p\big(\pi^*(g\cdot\gamma)\big)=\sum_{p\in\pi^{-1}[\{q\}]} \mathrm{Res}_p\big(\pi^*(g\cdot \gamma)\big)=\sum_{p\in\pi^{-1}[\{q\}]} \mathrm{Res}_p\big(\pi^*(g\cdot \alpha)\big).
\]
Here, the first equality follows because of $\Omega_{X'}(\mathcal{S}')=\pi_*\Omega_X(\mathcal{S})$ on $U_q\setminus\{q\}$. The last term is the pairing between $\alpha$ and $\varphi(g)$.

The commutativity between $f_2$ and $f_3$ is shown analogously.
\end{proof}
For $q\in X'$, let $n_q:=\dim(\bar{\mathcal{O}}_q/\mathfrak{c}_q)$ with the annihilator $\mathfrak{c}_q$~\eqref{eq:ann} of $\bar{\mathcal{O}}_q/\mathcal{O}_q$. Clearly, $n_q$ vanishes for $q\not\in S'$. Similarly to $\delta_q$, $n_q$ measures the order of the singularity $q\in S'$. In the following proposition, we compare these two integers.
\begin{Proposition}
\begin{enumerate}[(a)]
\item The annihilator $\mathrm{Ann}_q(\Omega_{X'}/\pi_*\Omega_X)$ equals $\mathfrak{c}_q$~\eqref{eq:ann} for $q\in X'$.
\item For $q\in S'$, one has $\delta_q+1\le n_q\le 2\delta_q$, and $n_q=2\delta_q$ holds if and only if $\Omega_{X',q}$ is a free $\mathcal{O}_q$-module of rank $1$.
\end{enumerate}
\end{Proposition}
\begin{proof}
\emph{(a)} Let $g\in\mathfrak{c}_q$. Then $g\cdot f\in\mathcal{O}_q$ for all $f\in\bar{\mathcal{O}}_q$. This implies that the pairing \eqref{eq:Res} of $gf$ with $\omega$ and therefore also the pairing of $f$ with $g\omega$ vanishes for all $f\in \mathcal{O}_q$ and all $\omega\in\Omega_{X',q}$. Due to Lemma~\ref{L:nondegenerate}, $g\cdot \omega\in (\pi_*\Omega_X)_q$ for all $\omega \in \Omega_{X',q}$, so $g\in \mathrm{Ann}_q(\Omega_{X'}/\pi_*\Omega_X)$.

Conversely, for $g\in \mathrm{Ann}_q(\Omega_{X'}/\pi_*\Omega_X)$, it is $g\cdot\omega\in(\pi_*\Omega_X)_q$ for all $\omega\in \Omega_{X',q}$. Again, the pairing in \eqref{eq:Res} of $f$ with $g\omega$ and of $gf$ with $\omega$ vanishes for all $\omega\in\Omega_{X',q}$ and all $f\in\mathcal{O}_q$. Due to Lemma~\ref{L:nondegenerate}, $g\cdot f\in\mathcal{O}_q$ for all $f\in\bar{\mathcal{O}}_q$ and hence $g\in \mathfrak{c}_q$.

\emph{(b)} 
From $\mathbb{C}+\mathfrak{c}_q\subset \mathcal{O}_q$, see \eqref{eq:inclusions}, we get $\bar{\mathcal{O}_q}/\mathcal{O}_q+\mathbb{C}\subset\bar{\mathcal{O}}_q/\mathfrak{c}_q$ and thus $1+\delta_q\le n_q$.
Since $\mathfrak{c}$ is a sheaf of ideals of $\bar{\mathcal{O}}_{X'}$, it is an $\bar{\mathcal{O}}_{X'}$-module and hence the direct image with respect to $\pi$ of a sheaf $\tilde{\mathfrak{c}}$ of ideals of $\mathcal{O}_X$. Here, we use the fact that $\bar{\mathcal{O}}_{X',q}$ separates the points in $\pi^{-1}[\{q\}]$. The sheaf $\tilde{\mathfrak{c}}$ is a coherent subsheaf of $\mathcal{O}_X\subset\mathcal{M}_X$ and therefore a generalised divisor on $X$. Since $X$ is smooth, there exists a divisor $D\ge 0$ on $X$ with $\tilde{\mathfrak{c}}=\mathcal{O}_{-D}$. The support of $D$ is contained in $S$ and we write $D=\sum_{p\in S}n_p\cdot p$. In particular, $n_q=\sum_{p\in\pi^{-1}[\{q\}]}n_p$. Because of (a), there exists for every $p\in \pi^{-1}[\{q\}]$ a form $\omega_p\in\Omega_{X',q}$ such that $\pi^*\omega_p$ has a pole of order $n_p$ at $p$. Otherwise, $\pi_*\mathcal{O}_{-D}\subsetneq\mathfrak{c}$. Let $\omega$ be a suitable linear combination of the $\omega_p$'s such that $\pi^*\omega$ has a pole of order $n_p$ at all $p\in S$. Due to (a), the map
\begin{equation}\label{eq:conductor-f}
\mathcal{O}_q/\mathfrak{c}_q\to \Omega_{X',q}/(\pi_*\Omega_X)_{q}\;,\qquad f\mapsto f\cdot\omega
\end{equation}
is injective. The alternating sum of dimensions of the exact sequence
\[
0\to\mathcal{O}_q/\mathfrak{c}_q\to\bar{\mathcal{O}}_q/\mathfrak{c}_q\to\bar{\mathcal{O}}_q/\mathcal{O}_q\to 0
\]
vanishes and so $\dim(\mathcal{O}_q/\mathfrak{c}_q)=n_q-\delta_q\le\delta_q$, i.e. $n_q\le 2\delta_q$. 

For $n_q=2\delta_q$, the map \eqref{eq:conductor-f} is an isomorphism. So every $\alpha \in \Omega_{X',q}$ can be represented as $\alpha=f\cdot\omega+\beta$ with $f\in\mathcal{O}_q,\beta\in(\pi_*\Omega_X)_q$. By definition of $\omega$, $\beta$ can be written as $g\cdot \omega$ with $g\in\mathfrak{c}_q\subset\mathcal{O}_q$. This implies that $\Omega_{X',q}$ has rank $1$ and is a free $\mathcal{O}_q$-module.

Conversely, let $\Omega_{X',q}$ be a free $\mathcal{O}_q$-module. Then, there exists one generator of this module. The pullback of this generator via $\pi$ has a pole of order $n_p$ at all $p\in S$ and no other poles. Otherwise, $\mathrm{Ann}(\Omega_{X'}/\pi_*\Omega_X)\subsetneq\pi_*\mathcal{O}_{-D}$. Therefore, we can choose $\omega$ in \eqref{eq:conductor-f} to be this generator and \eqref{eq:conductor-f} is surjective. This implies $n_q=2\delta_q$.
\end{proof}


\begin{Corollary}
If \,$X'$\, is of arithmetic genus \,$g'$\, and \,$\Sss'$\, a generalised divisor on \,$X'$\,, then:
\begin{enumerate}[(a)]
\item
\,$\deg(\Omega_{X'})=2g'-2$\,.
\item 
\,$\deg(\Omega_{X'}(\Sss'))=2g'-2-\deg(\Sss')$\,.
\item
If \,$\deg(\Sss')> 2g'-2$\,, then \,$H^1(X',\Sss') $\, is trivial.
\end{enumerate}  
\end{Corollary}

\begin{proof}
\emph{(a)}
For \,$\Sss'=\Oss_{X'}$\, we obtain by Serre duality (Theorem~\ref{T:serre})
\begin{align*}
g'& =\dim(H^1(X',\Oss_{X'}))=\dim(H^0(X',\Omega_{X'})) \;, \\
1 & = \dim(H^0(X',\Oss_{X'})) = \dim(H^1(X',\Omega_{X'}))  \;. 
\end{align*}
The Riemann Roch Theorem \ref{T:riemannroch} now implies the claimed statement.

\emph{(b)}
%
%
We first note that  \,$\deg(\Omega_X(\Sss))= \deg(\Omega_X)-\deg(\Sss)$\,, see \cite[\S 17]{Fo}, \,$\deg(\pi_* \Sss) = \deg(\Sss)+\delta$\, and \,$\deg(\pi_* \Omega_X(\Sss)) = \deg(\Omega_X(\Sss))+\delta$\,, see the left-hand side equation of \eqref{eq:deg-deg}. The non-degeneracy of the pairing in \eqref{eq:serre-pairing} implies
$$ \deg(\Omega_{X'}(\Sss'))- \deg(\pi_*\Omega_X(\Sss)) = \deg(\pi_*\Sss)-\deg(\Sss') 
$$
and hence,
\begin{align*}
\deg (\Omega_{X'}(\Sss'))
& = \deg(\pi_*\Omega_X(\Sss)) + \deg(\pi_*\Sss)- \deg(\Sss') \\
& = \deg (\Omega_X(\Sss))+\delta+\deg(\Sss)+\delta-\deg(\Sss') \\
& = \deg(\Omega_X)+2\delta-\deg(\Sss') \\
& = 2g-2+2\delta-\deg(\Sss') = 2g'-2-\deg(\Sss') \; . 
\end{align*}

\emph{(c)} 
Because the global meromorphic functions on \,$X'$\, have degree \,$0$, for every generalised divisor \,$\tilde{\Sss}'$\, on \,$X'$\, of negative degree, \,$H^0(X',\tilde{\Sss}')$\, is trivial. In particular, if \,$\deg(\Sss') > 2g'-2$\,, then \,$\deg(\Omega_{X'}(\Sss')) = \deg(\Omega_{X'}) - \deg(\Sss') < 0$\, and therefore, \,$H^0(X',\Omega_{X'}(\Sss'))=0$\,, whence \,$H^1(X',\Sss')=0$\, follows by Serre duality (Theorem~\ref{T:serre}).
\end{proof}

\section{The Krichever construction}\label{se:krichever}
As preparation for the construction of Baker-Akhiezer functions, we need a certain presentation of the elements of \,$H^1(X',\Oss_{X'})$\, by Mittag-Leffler distributions with support at given marked points. This representation is known as the Krichever construction, see \cite{Krichever}.

In the following, we denote by $H$ the algebra of germs of functions that are holomorphic in a punctured neighborhood of $0\in \mathbb{C}$.
We define the subsets
\begin{align*}
&H^+ = \{ h \in H \mid h \text{ extends holomorphically to } 0\} = \C\{z\} \\
&H^- = \{ h \in H \mid h \text{ extends holomorphically to } 
\P\backslash\{0\} \text{ with } h(\infty)=0\} \\
&H^-\ind{finite} = \{ h \in H^- \mid h \text{ has a pole at }0\} = z^{-1} \cdot \C[z^{-1}]. 
\end{align*}
There is a decomposition analogous to the Birkhoff factorization:

\begin{Lemma}
	$H = H^+ \oplus H^-$.
\end{Lemma}
\begin{proof}
	For any $h\in H$, let $h^+(z)=\frac{1}{2\pi\imath}\oint\frac{h(z')dz'}{z'-z}$. Here, the integral is taken along a path in the domain of definition of $h$ around $z$ and $0$ in the anti-clockwise order. Moreover, let $h^-(z)=\frac{1}{2\pi\imath}\oint\frac{h(z')dz'}{z'-z}$. This time the integral is taken along a path in the domain of definition of $h$ around $0$, but not around $z$, in the clockwise order. Since the form $\frac{h(z')dz'}{z'-z}$ is closed, these integrals do not depend on the choice of the path of integration. Due to Cauchy's integral formula, we have $h=h^++h^-$. Moreover, $h^+$ is holomorphic in a neighbourhood of $0$ and $h^-$ is holomorphic on $\P\setminus\{0\}$ and vanishes at $\infty$.
	
	The intersection $H^+\cap H^-$ contains holomorphic function on $\P$ which are constant. Since they vanish at $\infty$, they are identically zero.
\end{proof}

Let $X'$ be a compact singular curve with marked smooth points $q_1,\ldots,q_n \in X'$ and holomorphic charts $z_1,\ldots,z_n$ centered at those points, i.e.\ $z_i(q_i)=0$. 

For any $(h_1,\ldots,h_n)\in H^n$, one can choose disjoint smooth open neighborhoods $U_1,\ldots,U_n$ of $q_1,\ldots,q_n$ such that $z_i^\ast h_i$ is defined on $U_i\backslash\{q_i\}$. Together with $U_0 = X'\backslash\{q_1,\ldots,q_n\}$, we get a cover $\mathcal{U} =\{ U_0,U_1,\ldots,U_n\}$  of $X'$. Since the only non-empty pairwise intersections are of the form $U_0\cap U_i=U_i \setminus\{0\}$, $z_i^\ast h_i$ is holomorphic on these intersections and defines an element of \,$C^1(\mathcal{U},\Oss_{X'})$\,. Because the intersection of every triple of distinct \,$U_i$\, is empty, we have \,$C^2(\mathcal{U},\Oss_{X'})=0$\, and hence, \,$(h_1,\dotsc,h_n)$\, defines a cocycle and induces an element of \,$H^1(\mathcal{U},\Oss_{X'})$\,. 
Since $\mathcal{U}$ is a Leray cover of $X'$ for $\Oss_{X'}$ as in the proof of Lemma~\ref{L:dimH}, we obtain $H^1(\mathcal{U},\Oss_{X'})=H^1(X',\Oss_{X'})$ and the surjective map
\begin{align}\label{eq:krichever map}
\varphi:H^n&\to H^1(X',\Oss_{X'}) \;.
\end{align}
The Serre Duality theorem \ref{T:serre} states that 
\begin{align}\label{eq:serre duality}
\bigr( \, \varphi(h_1,\ldots,h_n)\,,\, \omega \,\bigr) & \mapsto 
\sum_{i=1}^n \Res_{q_i} z_i^\ast h_i\,\omega
\end{align}
defines a non-degenerate pairing between  $\varphi(h_1,\ldots,h_n)\in H^1(X',\Oss_{X'})$
and $\omega\in H^0(X',\Omega_{X'})$.

If $\varphi(h_1,\ldots,h_n)$ is the coboundary of an element of $C^0(\mathcal{U},\Oss_{X'})$, then for all $\omega\in H^0(X',\Omega_{X'})$, the sum of the residues at $q_1,\ldots,q_n$ of the product $\omega f$ with the corresponding holomorphic function $f:U_0\to\mathbb{C}$ vanishes. Therefore, \eqref{eq:serre duality} indeed vanishes on all elements in the kernel of $\varphi$ and defines a pairing between $H^0(X',\Omega_{X'})$ and $H^1(X',\Oss_{X'})$.

%

Each element $(h_1,\ldots,h_n)\in (H^-\ind{finite})^n$ defines a Mittag-Leffler distribution on $X'$. A solution is a meromorphic function $f$ on $X'$ with the same principal parts at $q_1,\ldots,q_n$, i.e.\ $f-z_i^\ast h_i$ is holomorphic on $U_i$, and \,$f$\, is holomorphic on \,$U_0$\,. The following lemma implies that the distribution $(h_1,\ldots,h_n)\in (H^-\ind{finite})^n$ has a solution if and only if $\varphi(h_1,\ldots,h_n)=0$, i.e.\ if $\sum_{i=1}^n \Res_{q_i} z_i^\ast h_i \,\omega=0$ for all $\omega \in H^0(X',\Omega_{X'})$.

\begin{Lemma}\label{L:phi kernel image}
	\begin{itemize}
		\item[(i)] The kernel of $\varphi$ is equal to the set of those elements $(h_1,\ldots,h_n)$ $\in H^n$ that admit a holomorphic function $f$ on $U_0$ such that  $f-z_i^\ast h_i$ is holomorphic at $0$. In particular, $(H^+)^n$ is contained in the kernel of $\varphi$.
		\item[(ii)] The restriction of \,$\varphi$\, to  $(H^-\ind{finite})^n$ is surjective.
	\end{itemize}
\end{Lemma}
\begin{proof}
	\emph{(i)} 
	This follows from the fact that $(h_1,\ldots,h_n)\in H^n$ is in the kernel of $\varphi$ if and only if the cocycle defined by $z_i^\ast h_i$ is a coboundary. 
	
	\noindent\emph{(ii)}
	 We first conclude $\varphi((H^-)^n)=H^1(X',\Oss_{X'})$ 
from the inclusion $(H^+)^n\subset \ker\varphi$ and the surjectivity of $\varphi$. The space $H^0(X',\Omega_{X'})$ is finite-dimensional. Therefore, there exists $N\in\mathbb{N}$ such that all non-trivial $\omega\in H^0(X',\Omega_{X'})$ vanish at any $q_1,\ldots,q_n$ at most to order $N$. For every \,$\omega \in H^0(X',\Omega_{X'})$\,, the pairing of \,$\omega$\, with \,$\varphi(h)$\, for some \,$h\in (H^-\ind{finite})^n$\, is non-zero by definition of~\eqref{eq:serre duality}. Thus, by the non-degeneracy of the pairing~\eqref{eq:serre duality},  the restriction of the map $\varphi$ to $(H^-\ind{finite})^n$ is surjective.
\end{proof}
From the long exact sequence of $0\rightarrow \mathbb{Z} \rightarrow \Oss_{X'}{\xrightarrow{e^{2\pi i\cdot}}}\Oss_{X'}^\ast\rightarrow 1$, the exact sequence \,$0 \to \Z \to \C \to \C^* \to 0$\, can be split off and thereby, we obtain the exact sequence
\[ 0\rightarrow H^1(X',\mathbb{Z}) \rightarrow H^1(X',\Oss_{X'})
\rightarrow  H^1(X',\Oss_{X'}^\ast) \rightarrow H^2(X',\mathbb{Z})
\rightarrow 0\,, \]
where the map $H^1(X',\Oss_{X'})\rightarrow  H^1(X',\Oss_{X'}^\ast)$ is induced by $\exp(2\pi i\,\cdot\,)$ and the connecting homomorphism $H^1(X',\Oss_{X'}^\ast)\to H^2(X',\mathbb{Z}) \cong \Z$ is the degree map of invertible sheaves. We mention that $H^1(X',\Oss_{X'}^\ast)$ is the Picard group \,$\Pic(X')$\,, i.e.\ the space of isomorphy classes of inver- tible sheaves. The map $H^1(X',\Oss_{X'})\rightarrow  H^1(X',\Oss_{X'}^\ast)$ can be considered as the exponential map of the Lie group \,$\Pic(X')$\,. 

In particular, each $h=(h_1,\ldots,h_n)\in (H^-\ind{finite})^n$ defines a one-parame- ter group 
of cocycles
$z_i^\ast\exp(2\pi i \,t \,h_i)$, where \,$t \in \C$\,.
Since these cocycles are from the \v{C}ech cohomology, one has $[g_\ell]\in H^1(X',\Oss)\simeq H^1(\Uss, \Oss)$.   
For given \,$t\in \C$\,, these cocycles map the holomorphic functions on \,$U_0$\, to holomorphic functions on \,$U_i \setminus \{q_i\}$\,. Transferring the construction of line bundles on smooth Riemann surfaces by cocycles as in \cite[Theorem~29.7]{Fo} yields a holomorphic line bundle on \,$X'$\, whose local sections define a locally free rank \,$1$\, sheaf \,$\Lss_h(t)$\, on \,$X'$\,.  

To show that \,$\Lss_h(t)$\, has a global meromorphic section, we also construct the corresponding line bundle on the normalisation \,$X$\,.  
The only non-empty intersections of two elements of \,$\mathcal{U}$\, are \,$U_0 \cap U_i = U_i \setminus \{q_i\}$\, and therefore smooth. Hence, they can be considered as subsets of \,$X$\,. Therefore, these cocycles together with the trivial line bundle on \,$\pi^{-1}[U_0]$\, define a line bundle on \,$X$\,.
This line bundle has a global meromorphic section, see \cite[Theorem~29.16]{Fo}. 
By our identification of meromorphic functions on \,$X$\, and on \,$X'$\,, this section is also a global meromorphic section of \,$\Lss_h(t)$\,.  By means of this section, we can identify \,$\Lss_h(t)$\, with a generalised divisor which is locally free.

This family is an one-parameter group, i.e.\ $\Lss_h(t+t') =\Lss_h(t)\otimes \Lss_h(t')$ for \,$t,t'\in \C$\,, where \,$\otimes$\, denotes the product in \,$\Pic(X')$\,. Because we have \,$\Lss_h(0)=\Oss_{X'}=\unity_{\Pic(X')}$\,, \,$\Lss_h(t)$\, stays in the unit component $\Pic_0(X')$ of \,$\Pic(X')$\,. \,$\Pic_0(X')$\, is the group of isomorphy classes of degree $0$ bundles. Conversely, every one-parameter group in $\Pic_0(X')$ is obtained that way because \,$\varphi$\, is surjective. This method for construc- ting linear flows on the Picard group is called
\emph{Krichever construction}\index{Krichever construction}. 

\begin{Lemma}\label{Lem:trivial_periodic_flow}
	\begin{itemize}
		\item[(i)] An element $h=(h_1,\ldots,h_n)\in (H^-\ind{finite})^n$ induces the trivial flow, i.e.\ $\Lss_h(t)=\unity_{\Pic(X')}$ for all $t\in \mathbb{C}$ if and only if the corresponding Mittag-Leffler distribution is solvable.  
		\item[(ii)] An element $h=(h_1,\ldots,h_n)\in (H^-\ind{finite})^n$ induces a periodic flow with period \,$T>0$\,, i.e.\ $\Lss_h(T)=\unity_{\Pic(X')}$ if and only if the Mittag-Leffler distribution can be solved by means of a multi-valued function $k$ whose values over a point differ by an element of $\tfrac{1}{T}\cdot \mathbb{Z}$ (i.e.\ $dk$ is an Abelian differential of the second kind with $\int_\gamma dk \in \tfrac{1}{T} \cdot \mathbb{Z}$ for all $\gamma \in H_1(X',\mathbb{Z})$).
	\end{itemize}
\end{Lemma}
\begin{proof}
	\emph{(i)} $\Lss_h(t)$ is trivial for all $t\in \mathbb{C}$ if and only if 
	\[
	\exp(2\pi i\,t\,
	\varphi(h)) = \unity_{\Pic(X')}\;. 
	\] 
	By taking the derivative with respect to \,$t$\, at \,$t=0$\,, we see that this is equivalent to $h\in \ker\varphi$. Lemma~\ref{L:phi kernel image}~(i) implies that $h\in\ker\varphi$ if and only if the Mittag-Leffler distribution admits a solution. 
	
	\noindent\emph{(ii)} 
	Using the preceding lemma again, we see that $\Lss_h(T)=\unity_{\Pic(X')}$ is equivalent to the existence of functions $k_0,\ldots,k_n$ on $U_0,\ldots,U_n$, respectively, with $k_0/k_i= z_i^\ast\exp(2\pi i \,T\, h_i)$. The multi-valued meromorphic function $k=\frac1{2\pi i T} \ln k_0$ has the desired properties.
\end{proof}

The Krichever construction is a key ingredient in the study of integrable systems and their relation to the theory of singular complex curves. Many integrable systems are determined by the following data: A compact singular curve $X'$, which is called the \emph{spectral curve}\index{spectral curve}, with smooth marked points $q_1,\ldots,q_n$ and holomorphic charts $z_1,\ldots,z_n$ centered at those points and two elements $h_1$ and $h_2\in (H^-\ind{finite})^n$. We distinguish between the following three cases: 

\begin{description}
	\item[Case 1] Both flows induced by  $h_1$ and $h_2$ are trivial.  
	\item[Case 2] $h_1$ induces a trivial and $h_2$ a periodic flow. 
	\item[Case 3] Both flows induced by $h_1$ and $h_2$ are periodic. 
\end{description}
The Lax operators corresponding to the integrable systems of Case~1 are meromorphic matrices, see~\cite{AMV}, whereas the systems of Case~2 and Case~3 are infinite-dimensional. In Case~2, the Lax operators are ordinary differential operators and in Case~3, they are partial differential operators. 

The periodic Korteweg de Vries equation belongs to Case~2. It is obtained for $n=1$, $h_1 = 1/z^2$ and $h_2= 1/z$. The corresponding Lax operator is the one-dimensional Schr\"odinger operator~\cite{La}. 

The periodic non-linear Schr\"odinger equation also belongs to Case~2. It is obtained for $n=2$, $h_1 = (1/z,1/z)$ and $h_2=(i/z,-i/z)$. The corresponding Lax operator is the one-dimensional Dirac operator with potentials~\cite{Sch}. 

An example for a system of Case~3 is the system of the double periodic Kadomcev-Petviashvilli equation \cite[Chapter 4]{FKT}. Here, we have $n=1$, $h_1 = 1/z$ and $h_2=2\pi i/z^2$. The Lax operator in this case is the $(1+1)$-dimensional heat operator, see Remark \ref{R:BA}\,(3).

%
\section{Baker-Akhiezer Functions}\label{se:BA}
There are two equivalent concepts to describe line bundles on complex curves: divisors and cocycles. Usually, only one of these is used. Baker-Akhiezer functions combine both concepts to describe sections on families of line bundles, see \cite[Chapter 2, \S 2]{DKN} for smooth curves. 
In this section, we define Baker-Akhiezer functions for general complex curves. 

Baker-Akhiezer functions are uniquely determined by their function-theoretic properties. The following two lemmata are used in the proof of uniqueness. The first is an easy consequence of Serre duality. We give a direct proof.
\begin{Lemma}\label{lem:serre_generalised}
Let $\Sss'\supset \Sss$ be two generalised divisors on $X'$. Then $H^1(X',\Sss)=0$ implies $H^1(X',\Sss')=0$.
\end{Lemma}
%
\begin{proof}
Because of $H^1(X',\Sss'/\Sss)=0$, the long exact sequence associated to the short exact sequence
\[
0 \to \Sss \to \Sss' \to \Sss'/\Sss \to 0
\]
is given by
\[
0\hspace{-1mm}\to\hspace{-1mm}H^0(X'\!,\Sss)\hspace{-1mm}\to\hspace{-1mm}H^0(X'\!,\Sss')\hspace{-1mm}\to\hspace{-1mm}H^0(X'\!,\Sss'\!/\!\Sss)\hspace{-1mm}\to\hspace{-1mm}H^1(X'\!,\Sss)\hspace{-1mm}\to\hspace{-1mm}H^1(X'\!,\Sss')\hspace{-1mm}\to\hspace{-1mm}0.
\]
Since by hypothesis $H^1(X',\Sss)=0$, it follows $H^1(X',\Sss')=0$.
\end{proof} 
\begin{Lemma}\label{lem:uniqueness}
Let \,$g'$\, be the arithmetic genus of \,$X'$\,,
\,$q_1,\dotsc,q_n \in X' \setminus S'$\, be pairwise different smooth points and \,$\Sss'$\, be a generalised divisor on $X'$ of degree $g'+n-1$ with $\supp(\Sss')\subset X'\setminus\{q_1,\dots, q_n\}$. Then, the linear map
\begin{equation}
\label{eq:H0-mapping}
	H^0(X',\Sss')\to \C^n, \qquad \varphi\mapsto (\varphi(q_1),\dots, \varphi(q_n)) 
\end{equation}
	is an isomorphism if and only if
\begin{equation} 
\label{eq:nonspecial}
	H^0(X',\Sss'_{-q_1-\dots-q_n}) =0,
\end{equation}
where \,$\Sss'_{-q_1-\dots-q_n}$\, is the generalised divisor obtained by multiplying \,$\Sss'$\, with the classical divisor \,$-q_1-\dotsc-q_n$\,. 
\end{Lemma}
\begin{proof}
If the map \eqref{eq:H0-mapping} is an isomorphism, its kernel $H^0(X',\Sss'_{-q_1-\dots-q_n})$ is zero. 

Conversely, we now suppose that $H^0(X',\Sss'_{-q_1-\dots-q_n})=0$ holds. So the map \eqref{eq:H0-mapping} is injective.  
Because of $\deg \Sss'=g'+n-1$, one has $\deg \Sss'_{-q_1-\dots-q_n}=g'-1$. Hence, the  Riemann-Roch Theorem~\ref{T:riemannroch} implies 
\begin{multline*}
		\dim H^0(X',\Sss'_{-q_1-\dots-q_n}) - \dim H^1(X',\Sss'_{-q_1-\dots-q_n})= \\
		= \deg \Sss'_{-q_1-\dots-q_n} +1-g'=0
\end{multline*}
		and therefore, \,$\dim H^1(X',\Sss'_{-q_1-\dots-q_n}) = \dim H^0(X',\Sss'_{-q_1-\dots-q_n}) = 0$\,. 
So Lemma \ref{lem:serre_generalised} gives \,$\dim H^1(X',\Sss')=0$\, and the Riemann-Roch Theorem yields \,$\dim H^0(X',\Sss')=n$\,. Finally, the map \eqref{eq:H0-mapping} is an isomorphism.
\end{proof}
In the case \,$n=1$\, and \,$\Oss_{X'} \subset \Sss'$, the condition $H^0(X',\Sss'_{-q_1})=0$ already implies that $\supp(\Sss')\subset X'\setminus \{q_1\}$. In fact, otherwise \,$1 \in H^0(X',\Sss'_{-q_1})$\,. For \,$n\geq 2$\,, the condition $H^0(X',\Sss'_{-q_1-\dots-q_n})=0$ only excludes the possibility that all points \,$q_1,\dotsc,q_n$\, are contained in \,$\supp(\Sss')$\,. 
\begin{Remark}\label{R:littlebitmoresingular}
The analogy of the case \,$n>1$\, to the original case \,$n=1$\, is elucidated by replacing \,$X'$\, with the more singular curve \,$X''$\, obtained from \,$X'$\, by identifying \,$q_1,\dotsc,q_n$\, to an ordinary \,$n$-fold point with \,$\delta_{X''}=\delta_{X'}+(n-1)$\,. 
The degree of the corresponding generalised divisor \,$\Sss''$\, on \,$X''$\, equals the arithmetic genus \,$g''$\, as in the case \,$n=1$\,. 
\end{Remark}
We now let \,$g'$\, be the arithmetic genus of \,$X'$\, and fix pairwise different smooth points \,$q_1,\dotsc,q_n \in X' \setminus S'$\,,
a generalised divisor \,$\Sss'$\, on \,$X'$\, of degree \,$g'+n-1$\, with \,$\Oss_{X'} \subset \Sss'$\, and \,$\mathrm{supp}(\Sss') \subset X' \setminus \{q_1,\dotsc,q_n\}$\, and \,$h_{\ell}  = (h_{\ell,1},\dotsc,h_{\ell,n})\in (H^-\ind{finite})^n$ for \,$\ell \in \{1,\dotsc,L\}$\,. In this setting, we will also use the notations from Section~\ref{se:krichever} associated with the Krichever construction and further suppose that \,$U_k \cap \mathrm{supp}(\Sss')  = \varnothing$\, for all \,$k\in \{1,\dotsc,n\}$\,. 
\begin{Lemma}
Every generalised divisor \,$\Sss'$\, of degree \,$g'+n-1$\, satisfying \eqref{eq:nonspecial} is equivalent to a generalised divisor \,$\Sss''$\, with \,$\Oss_{X'} \subset \Sss''$\,, \,$\mathrm{supp}(\Sss'') \subset X' \setminus \{q_1,\dotsc,q_n\}$\, and \,$\Sss''$\, satisfying \eqref{eq:nonspecial}.
\end{Lemma}
\begin{proof}
Because of \eqref{eq:nonspecial} and the Riemann-Roch Theorem \ref{T:riemannroch}, we have \,$H^1(X',\Sss'_{-q_1-\dotsc-q_n})=0$\, and therefore also \,$H^1(X',\Sss'_{-q_1-\dotsc-\widehat{q_k}-\dotsc-q_n})=0$\, for every \,$k\in \{1,\dotsc,n\}$\, by Lemma~\ref{lem:serre_generalised}, where \,$\widehat{q_k}$\, stands for the omission of the summand \,$q_k$\,. Again, by application of the Riemann-Roch Theorem \ref{T:riemannroch}, we obtain \,$\dim H^0(X',\Sss'_{-q_1-\dotsc-\widehat{q_k}-\dotsc-q_n})=1$\,. Let \,$f_k$\, be a non-zero element of \,$H^0(X',\Sss'_{-q_1-\dotsc-\widehat{q_k}-\dotsc-q_n})$\,. Then, \,$f := f_1+\dotsc+f_n\in H^0(X',\Sss') \setminus H^0(X',\Sss'_{-q_k})$\, for all \,$k\in\{1,\dotsc,n\}$\,. 

We define \,$\Sss'' := f^{-1}\cdot \Sss'$\,. Then, \,$\Sss''$\, is a generalised divisor of degree \,$g'+n-1$\,. Because of \,$f\in H^0(X',\Sss')$, we have \,$1 \in H^0(X',\Sss'')$\, and therefore \,$\Oss_{X'} \subset \Sss''$\,. Finally, because \,$q_k$\, is a smooth point and \,$f\not\in H^0(X',\Sss'_{-q_k})$\,, we have \,$q_k \not\in \supp(\Sss'')$\, for \,$k\in \{1,\dotsc,n\}$\,. 
\end{proof}
We abbreviate for $h_1,\ldots,h_L\in(H^-)^n$
$$ \Lss_h(t) := \Lss_{h_1}(t_1) \otimes \dotsc \otimes \Lss_{h_L}(t_L) \quad\text{for}\quad t=(t_1,\dotsc,t_L)\in \mathbb{C}^L $$
and define
\begin{equation}\label{eq:T}
	T := \{ t\in \mathbb{C}^L \mid
	\dim H^0(X,\Sss'_{-q_1-\dotsc-q_n}\otimes \Lss_{h}(t)) \neq 0\} \; . 
\end{equation}
We will see that \,$T$\, is a subvariety of \,$\mathbb{C}^L$\, as well as the set of parameters for which the Baker-Akhiezer function is not uniquely defined. For fixed \,$t\in \C^L$\,, we consider $\Sss' \otimes \Lss_{h}(t)$ as sheaf on \,$X'$\,. The same cocycles with variable \,$t\in \C^L$\, induce a sheaf \,$\Lss_h$\, on \,$X' \times \C^L$\,, and we will consider the sheaf \,$\Sss' \otimes \Lss_h$\, on \,$X' \times \C^L$\, (where we also regard \,$\Sss'$\, as a sheaf on \,$X' \times \C^L$\,).
\begin{Lemma}\label{flat sheaf}
The sheaf $\Sss' \otimes \Lss_{h}$ is flat with respect to the map \,$X' \times \mathbb{C}^L \to \mathbb{C}^L$\,. 
\end{Lemma}
\begin{proof}
In the proof, we involve the concepts described in \cite[Chapter II, \S 2]{GPR}. We first note that every vector space over \,$\C$\, is a flat \,$\C$-module by \cite[Chapter II, Proposition~2.1(1)]{GPR} because \,$\C$\ does not contain any non-trivial ideals. The simplest complex space \,$Y$\, is a single point with the sheaf of holomorphic functions equal to \,$\C$\,. Then, \,$\Sss'$\, is a \,$f$-flat \,$\Oss_{X'}$-module for the constant map \,$f: X' \to Y$\,. By \cite[Chapter II, Proposition~2.6(1)]{GPR} (applied with \,$X=X'$\,, \,$Y$\, as above and \,$Z=\C^L$\,, \,$f$\, as above and the constant map \,$g: \C^L\to Y$\,), \,$\Sss'$\, considered as sheaf on \,$X' \times \C^L$\, is flat with respect to the projection \,$X' \times \C^L\to \C^L$\,. Because \,$\Lss_h$\, is locally free on \,$X' \times \C^L$\,, the assertion follows.
\end{proof}
This Lemma shows that $\Sss' \otimes \Lss_h$ is a deformation of the sheaf $\Sss'$ on $X'$. Now, we can use the theory of deformations of sheaves and can control the dependence of the cohomology groups $H^q(X',\Sss'\otimes \Lss_h(t))$ on $t\in\mathbb{C}^L$.
\begin{Theorem}\label{S subvariety}
\,$T$\, is a subvariety of \,$\C^L$\,. 
\end{Theorem}
\begin{proof}
	In this proof, we involve the concepts described in \cite[Chapter III, \S 4.2]{GPR}.
	\,$T$\, is closed by \cite[Chapter III, Theorem~4.7(a)]{GPR}. 
	For $q>1$, the cohomology groups $H^q(X',\Sss'_{-q_1-\dotsc-q_L}\otimes \Lss_h(t))$ are trivial because there exists a Leray covering \,$\Uss$\, of \,$X'$\, such that the intersection of any triple of distinct members of \,$\Uss$\, is empty. Now, \cite[Chapter~III Theorem~4.7~(b)]{GPR} implies that for the sheaf $\Sss'_{-q_1-\dotsc-q_L}\otimes \Lss_h(t)$ and $q\ge 1$, base change holds. Furthermore, \cite[Chapter~III Corollary~4.8]{GPR} implies that \,$R^1f_*\Sss$\, is zero on the open set \,$\C^L\setminus T$\,, with $f:X'\times\mathbb{C}^L\to\mathbb{C}^L$, $f(x,t)=t$ and the sheaf $\Sh{S}=\Sss'_{-q_1-\dotsc-q_L}\otimes \Lss_h$. Therefore, \,$\supp(R^1f_*\Sss)$\, is contained in \,$T$\,.
	
    Conversely, since $X$ is compact, $f$ is proper and the sheaf $R^1f_\ast\Sh{S}$ is coherent on $\mathbb{C}$ due to Grauert's Direct Image Theorem \cite[Chapter~III Theorem~4.1]{GPR}. Therefore, \,$\supp(R^1f_*\Sss)$\, is a subvariety of \,$\C^L$\,. The complement of \,$\supp(R^1f_*\Sss)$\, is contained in the complement of \,$T$\, again by \cite[Chapter~III Corollary~4.8]{GPR}. This implies \,$\supp(R^1f_*\Sss)=T$\,.
\end{proof}
\begin{Definition}
\label{D:BA}
For \,$(c_{j,k})\in \mathrm{GL}(n,\C)$\,, 
a \emph{Baker-Akhiezer function} is a function
	\begin{equation*}
		\psi=(\psi_1,\dotsc,\psi_n):(X'\setminus\{q_1,\dotsc,q_n\})\times(\mathbb{C}^L\setminus T) \to\mathbb{C}^n,
	\end{equation*}
	with the following properties:
\begin{itemize}
	\item[(i)] The map $q\mapsto \psi_j(q,t)$ is a holomorphic section of $\Sss'$ on $ X\setminus\{q_1,\dotsc,q_n\}$ for \,$t\in \C^L\setminus T$\, and \,$j\in \{1,\dotsc,n\}$\,. 
	\item[(ii)] For \,$j,k\in \{1,\dotsc,n\}$\, and \,$t\in \C^L\setminus T$\,, the map
$$ q\mapsto\psi_j(q,t)\cdot z_k^\ast\exp\left(-2\pi i\, \sum_{\ell=1}^L t_{\ell}\, h_{\ell,k} \right) $$
    extends to a holomorphic function on $U_k$ with value \,$c_{j,k}$\, at \,$q_k$\,. 
\end{itemize}
\end{Definition}	
\begin{Theorem}
\label{T:BA}
For every \,$(c_{j,k})\in \mathrm{GL}(n,\C)$\,, there exists one and only one Baker-Akhiezer function \,$\psi$\, and \,$\psi \in H^0(U_0 \times (\C^L\setminus T), \Sss' \otimes \Lss_h)$\,.
\end{Theorem}
\begin{proof}
An element of $H^0(X',\Lss_h(t))$ is given by holomorphic functions $\varphi_k \in H^0(U_k,\Oss_{U_k})$ for \,$k\in\{0,\dotsc,n\}$\, such that on \,$U_0\cap U_k=U_k\setminus \{q_k\}$,
\begin{equation}
\label{eq:fk-trafo}
\varphi_0\cdot z_k^*\exp\left(-2\pi i \sum\limits_{\ell=1}^L t_\ell h_{\ell,k}\right)=\varphi_k \; . 
\end{equation}
Consequently, for $t\in \C^L\setminus T$, an  element \,$\varphi$\, of $H^0(X',\Sss' \otimes \Lss_h(t))$ is given by functions \,$\varphi_0 \in H^0(U_0,\Sss')$\ and $\varphi_k\in H^0(U_k,\Oss_{U_k})$ for $k\in \{1,\dots, n\}$ such that on \,$U_k \setminus \{q_k\}$\, again \eqref{eq:fk-trafo} holds.

By Lemma~\ref{lem:uniqueness}, for \,$t\in \C^L\setminus T$\,, the map  
\begin{equation*}
H^0(X',\Sss' \otimes \Lss_h(t))\to \C^n, \qquad \varphi\mapsto (\varphi(q_1),\dots, \varphi(q_n)) 
\end{equation*}
is an isomorphism. Therefore, for every \,$j\in \{1,\dotsc,n\}$, there exists a unique element  \,$\varphi \in H^0(X',\Sss' \otimes \Lss_h(t))$\, which is mapped to \,$(c_{j,1},\dotsc,c_{j,n})$\, by the above isomorphism. Then, \,$\psi_j=\varphi_0$\, is the unique function with the properties of the \,$j$-th component of the Baker-Akhiezer function \,$\psi$\, given in Definition~\ref{D:BA}.

It remains to show that \,$\psi$\, is holomorphic in \,$t\in \C^L \setminus T$\,. Because of the uniqueness of the Baker-Akhiezer function, it suffices to show this in the case \,$c_{j,k}=\delta_{jk}$\,. For \,$j\in \{1,\dotsc,n\}$, we define 
$$\Sss_j := \Sss'_{-q_1-\dotsc-\widehat{q}_j-\dotsc-q_n} \otimes \Lss_h
\quad\text{and}\quad
\Sss_j(t) := \Sss'_{-q_1-\dotsc-\widehat{q}_j-\dotsc-q_n} \otimes \Lss_h(t)  
$$
for \,$t\in \C^L \setminus T$\,. 
By Lemma~\ref{lem:uniqueness}, we have 
\,$\dim H^0(X',\Sss_j(t)) = 1$\,. Due to \cite[Chapter~III, Theorem~4.7(d)]{GPR}, \,$f_*\Sss_j$\, is locally free of rank \,$1$\,, where \,$f: X' \times (\C^L\setminus T) \to \C^L\setminus T, \; (x,t) \mapsto t$\,, and for every \,$t_0 \in \C^L \setminus T$, the values of holomorphic sections of \,$f_*\Sss_j$\, at \,$t_0$\, are canonically isomorphic to \,$H^0(X', \Sss_j(t_0)) $\,. In particular, there exists a holomorphic section \,$\varphi$\, of \,$f_* \Sss_j$\, that is normalised at \,$t_0$\,. In order to obtain \,$\psi_j$\,, we normalise \,$\varphi$\,. By definition of \,$\Sss_j$\,, the corresponding \,$\varphi_j$\, as above is holomorphic on \,$U_j \times (\C^L\setminus T)$\,. Evaluation at \,$q_j$\, yields a holomorphic function \,$\varphi_j(q_j)$\, on \,$\C^L\setminus T$\, which is equal to \,$1$\, at \,$t_0$\,. Because \,$\varphi_j(q_j)$\, does not vanish in a neighborhood of \,$t_0$\,, \,$\psi_j=\varphi_0 / \varphi_j(q_j)$\, is holomorphic there.  
\end{proof}
\begin{Remark}\label{R:BA}
\begin{enumerate}
\item
Baker-Akhiezer functions with \,$n=L=1$\, were first constructed by \textsc{Baker} in relation to the study of commuting pairs of ordinary differential operators, see \cite{Baker}. Later on, \textsc{Akhiezer} applied them to the investigation of the spectral theory of ordinary differential equations in \cite{Akhiezer}. Baker-Akhiezer functions with general \,$n$\, and \,$L$\, were first introduced by \textsc{Krichever}, \cite{Krichever}.
\item
If \,$X'=X$\, is smooth and \,$\Sss'$\, is a non-special, positive divisor of degree \,$g+n-1$\, with support in \,$X\setminus \{q_1,\dotsc,q_n\}$\,, one can express the corresponding Baker-Akhiezer function in terms of Riemann's theta function, see \cite[Chapter II, Theorem~2.2]{DKN}. 
\item
We illustrate with an example how to construct a Baker-Akhie\-zer function which solves the heat equation 
\begin{gather}
\label{eq:heatequation}
(\partial_y-\partial_x^2+u(x,y))\psi(x,y)=0
\end{gather}
with some (complex-valued) potential \,$u$\,, see \cite[Chapter~4]{FKT}.

We let an arbitrary compact singular curve \,$X'$\, with arithmetic genus \,$g'$\, and a marked smooth point \,$q_1 \in X'$\, (\,$n=1$\,) be given and choose \,$L=2$\,. Let \,$z$\, be a local coordinate centered at \,$q_1$\, and define
$$ h_1(z) := \frac{1}{z} \quad\text{and}\quad h_2(z) := \frac{2\pi i}{z^2} \; . $$
Further, we suppose that \,$\Sss'$\, is a generalised divisor of degree \,$g'$\, with support away from \,$q_1$\,. 

Let \,$\psi$\, be the Baker-Akhiezer function associated to these data, where \,$t = (x,y) \in \C^2 \setminus T$\,. Then, there exists an analytic function \,$u: \C^2 \setminus T \to \C$\, such that \,$\psi$\, solves \eqref{eq:heatequation}. 
\begin{proof}
With \,$\varphi_0=\psi$\, and \,$\varphi_1$\, as in the proof of Theorem~\ref{T:BA}, we have
$$ \varphi_0 = z^*\exp\bigr(2\pi i (x\cdot h_1 + y\cdot h_2)\bigr) \cdot \varphi_1 $$
and therefore
\begin{align*}
& \hspace{1.2cm}(\partial_y - \partial_x^2)\varphi_0
= z^* \exp\bigr(2\pi i (x\cdot h_1 + y\cdot h_2)\bigr) \\&\hspace{1.1cm}
\cdot \biggr( (2\pi i z^*h_2 + 4\pi^2 z^* h_1^2)\cdot \varphi_1 + (\partial_y - \partial_x^2)\varphi_1 - 4\pi i z^*h_1\cdot \partial_x\varphi_1 \biggr) \\&\hspace{1.1cm}
= z^* \exp\bigr(2\pi i (x\cdot h_1 + y\cdot h_2)\bigr) 
\cdot \biggr( (\partial_y - \partial_x^2)\varphi_1 - 4\pi i z^*h_1\cdot \partial_x\varphi_1 \biggr)
\end{align*}
because of the choice of \,$h_1$\, and \,$h_2$\,. The last factor on the right-hand side is holomorphic on \,$U_1$\, because \,$\partial_x \varphi_1$\, has a zero at \,$q_1$\, due to the normalisation \,$\varphi_1(q_1)=1$\,.  
Now, let \,$u(x,y)$\, denote the value of this function at \,$q_1$\,. \,$(\partial_y -\partial_x^2)\varphi_0$\, has the properties (i) and (ii) of Definition~\ref{D:BA} with the exception of the normalisation. The proof of Theorem~\ref{T:BA} shows that the elements of \,$H^0(X',\Sss' \otimes \Lss_h(x,y))$\, are uniquely determined by their values of the corresponding \,$\varphi_1$\, at \,$q_1$\,. 
Therefore, we have \,$(\partial_y -\partial_x^2)\varphi_0 = u(x,y)\cdot \varphi_0$\,.
\end{proof}
	
\end{enumerate}
\end{Remark}


\begin{thebibliography}{MMM}
\bibitem[A-vM-V]{AMV} M.\ Adler, P.\ van Moerbeke, P.\ Vanhaecke:
  Algebraic integrability, Painlev\'{e} geometry and Lie algebras,
\bibitem[Ak]{Akhiezer} N.~I.~Akhiezer: A continuous analogue of orthogonal polynomials on a system of intervals. Dokl.~Akad.~Nauk SSSR \textbf{141} (1961), 263--266. (English Translation: Sov.~Math., Dokl.~\textbf{2} (1961), 1409--1412.)
\bibitem[Ba]{Baker} H.~F.~Baker: Note on the foregoing paper "Commutative ordinary differential operators," by J.~L.~Burchnall and J.~W.~Chaundy. Proc.~R.~Soc.~Lond., Ser.~A \textbf{118} (1928), 584--593.
\bibitem[Br]{Bredon} G.\ E.\ Bredon: Topology and Geometry. Graduate Texts in Mathematics {\bf 139}. Springer, New York (1993).
\bibitem[dJ-P]{DJo} T.\ de Jong, G.\ Pfister: Local Analytic Geometry.
  Advanced Lectures in Mathematics. Vieweg, Braunschweig/Wiesbaden (2000).
\bibitem[D-K-N]{DKN} B.\ A.\ Dubrovin, I.\ M.\ Krichever, S.\ P.\
  Novikov: Integrable systems I. In: V.\ I.\ Arnold, S.\ P.\ Novikov
  (eds.) Dynamical Systems IV. Encyclopedia of Mathematical Sciences
  {\bf 4}, pp. 173-280. Springer, Berlin, Heidelberg (1990).
\bibitem[F-K-T]{FKT} J.\ Feldman, H.\ Kn\"orrer, E.\ Trubowitz:
  Riemann surfaces of infinite genus. CRM Monograph Series {\bf 20},
  American Math.\ Soc.\, Providence (2003).
\bibitem[Fo]{Fo} O.\ Forster: Lectures on Riemann surfaces. Graduate Texts in Mathematics {\bf 81}. Springer, New York (1981).
\bibitem[G-P-R]{GPR} H.\ Grauert, Th.\ Peternell, R.\ Remmert (eds.):
  Several complex variables VII. Encyclopedia of Mathematical Sciences
  {\bf 74}. Springer, Berlin, Heidelberg (1994).
\bibitem[G-L-S]{GLS} G.-H.\ Greuel, C.\ Lossen, E.\ Shustin:
  Introduction to Singularities and Deformations. Springer Monographs
  in Mathematics. Springer, Berlin (2007).
\bibitem[Ha]{Ha} R.\ Hartshorne: Algebraic geometry. Graduate Texts in Mathematics {\bf 52}. Springer, New York (1977).
\bibitem[Ha-86]{Ha86} R.\ Hartshorne: Generalized divisors on Gorenstein
  curves and a theorem of Noether. J.\ Math.\ Kyoto Univ.\ {\bf 26} No.\ 3,
  375-386 (1986).
\bibitem[Hi]{Hi} N.\ J.\ Hitchin: Harmonic Maps from a 2-Torus to the 3-Sphere. J.\ Differential Geometry {\bf 31}, 627-710 (1990).
\bibitem[Kr]{Krichever} I.~M.~Krichever: Integration of nonlinear equations by the methods of algebraic geometry. Funkts.~Anal.~Prilozh.~\textbf{11} (1977), 15--31. (English translation: Funct.~Anal.~Appl.~\textbf{11} (1977), 12--26.)
\bibitem[La]{La} P.\ D.\ Lax: Integrals of Nonlinear Equations of Evolutions and Solitary Waves. Commun.\ Pure Applied Math.\ {\bf 21}, 467-490 (1968).
\bibitem[Ro]{Ros} M.\ Rosenlicht: Equivalence relations on algebraic curves.
  Ann.\ of Math.\ {\bf 56}, 169-191 (1952).
\bibitem[Sch]{Sch} M.\ U.\ Schmidt: Integrable systems and Riemann surfaces
  of infinite genus. Memoirs of the American Math.\ Soc.\ {\bf 581} (1996).
\bibitem[Se]{Se} J.-P.\ Serre: Algebraic groups and class fields.
  Graduate Texts in Mathematics {\bf 117}, Springer, New York (1988).
\end{thebibliography}
\end{document}